\RequirePackage{fix-cm}
\documentclass[smallextended]{svjour3}       % onecolumn (second format)
\smartqed  % flush right qed marks, e.g. at end of proof
\usepackage{graphicx}
\usepackage{setspace}
\usepackage{amsmath}
\usepackage{amssymb}
\usepackage{verbatim}
\usepackage{hyperref}
\usepackage{afterpage}
\usepackage{wasysym}
\usepackage{tabmac_avi}
\usepackage{pgfplots}
\usepackage{enumitem}
\usepackage{cancel}

\newtheorem{prop}[theorem]{Proposition}
\newtheorem{corr}[theorem]{Corollary}

\newtheorem{ex}[theorem]{Example}

\begin{document}

\title{Statistical structure of concave compositions\thanks{Part of this research was conducted while the authors were graduate students at Drexel University. Some of the work on this project was funded under European Research Council under the European Union's Seventh Framework Programme (FP/2007-2013) / ERC Grant agreement n. 335220 - }
}
%\subtitle{Do you have a subtitle?\\ If so, write it here}

%\titlerunning{Short form of title}        % if too long for running head

\author{Avinash J. Dalal         \and
		Amanda Lohss \and Daniel Parry %etc.
}

%\authorrunning{Short form of author list} % if too long for running head

\institute{Avinash J. Dalal  \at
              Department of Mathematics and Statistics, University of West Florida, Pensacola, FL 32514, USA \\
              \email{ adalal78@gmail.com}           %  \\
%             \emph{Present address:} of F. Author  %  if needed
           \and
           Amanda Lohss \at
           Department of Computing, Mathematics, and Physics, Messiah University, Mechanicsburg, 
           PA  17055, USA
               \email{aglohss@gmail.com} 
               \and
               Daniel Parry \at
               Model Risk Governance and Review, JP Morgan Chase, 4 Chase, MetroTech Center 11245 New York, New York
               \email{daniel.parry@jpmchase.com}
}

\date{Received: date / Accepted: date}
% The correct dates will be entered by the editor

\maketitle

\begin{abstract}
	In this paper, we study concave compositions, an extension of partitions that were considered by Andrews, Rhoades, and Zwegers. They presented several open problems regarding the statistical structure of concave compositions including the distribution of the perimeter and tilt, the number of summands, and the shape of the graph of a typical concave composition. We present solutions to these problems by applying Fristedt's conditioning device on the uniform measure.  
	
\keywords{Partitions \and concave compositions \and limit shape}
% \PACS{PACS code1 \and PACS code2 \and more}
 \subclass{05A16 \and  60C05 \and 11P82}
\end{abstract}

\section{Introduction}
\label{intro}
%----------------------------------------------------------------------------------
A composition of a positive integer $n$ is a finite sequence of positive integers which sum to $n$.  The study of compositions dates back to MacMahon \cite{MR2417935}, where he made significant contributions to plane partitions, a particular subset of compositions, the Rogers-Ramanujan identities and partition analysis.  For more on the history of compositions see the book of Heubach and Mansour \cite{MR2531482}.  There are many different types of compositions which are studied such as Carlitz compositions \cite{MR1924786} and their generalizations \cite{MR2180794}, stacks \cite{MR0229604,MR0282940,MR0299575}, unimodal sequences \cite{MR2994899}, and partitions \cite{MR1634067}.

One general form of compositions are concave compositions, which can be thought of as the convolution of two random partitions.  In \cite{And11}, Andrews studies these compositions of even length, where their generating function is derived through the pentagonal number theorem, and the false theta function reveals new facts about concatenatable, spiral and self-avoiding walks (CSSAWs).  In \cite{And13}, Andrews links the generating function of concave compositions to a fusion of classical, false, and mock theta functions as well as other Appell-Lerch sums.  More recently, in \cite{MR3152010}, Andrews, Rhodes and Zwegers asked several questions regarding the statistical structure of concave compositions, including the following.
%In CITE, Andrews studied a more general form of compositions known as concave compositions, which can be thought of as the convolution of two random partitions.  Andrews, Rhoades and Zwegers \cite{MR3152010} studied a more general form of compositions known as concave compositions which can be thought of as the convolution of two random partitions. In their paper, several questions were asked regarding the statistical structure of concave compositions, including the following.
\begin{enumerate}
	\item What is the distribution of the perimeter of a concave composition?
	\item How many summands are there for a typical concave composition?
	\item What is the distribution of the tilt in a concave composition?
	\item What is the {\it{typical shape}} of the graph of a concave composition?
\end{enumerate}

The goal of this paper will be to demonstrate solutions to these questions, and in that regard we organize the paper as follows. In Section~\ref{prelim}, we introduce the necessary definitions and notation. In Section~\ref{BMsec} we apply Fristedt's conditioning device, as employed in \cite{MR1667320,MR2422389,MR2915644} on the uniform measure with respect to concave compositions. In Section~\ref{PTLsec}, the distributions of the perimeter, tilt, and summands of a typical concave composition are derived. Finally Section~\ref{graphsec} discusses the typical shape of the graph of a concave composition.

%------------------------------------------------------------------------------------------------
\section{Preliminaries}
\label{prelim}
%------------------------------------------------------------------------------------------------
A concave composition of a positive integer $n$ is a sequence of integers $\lambda^{-}_1 \geq \lambda^{-}_2 \geq \lambda^-_3 \geq \cdots \geq \lambda^-_L> c < \lambda^+_1 \leq \lambda^+_2\leq \lambda^+_3 \leq \cdots  \leq \lambda^+_R\,,$ where
\begin{equation}\label{ccdef}
\sum_{i=1}^L \lambda^-_i + c + \sum_{j=1}^R \lambda^+_j = n,
\end{equation}
and $c \geq 0$ is the {\it central part} of the composition.

In \cite{MR3152010} a concave composition was expressed in terms of two partitions and the central part. A partition $\lambda=(\lambda_1,\lambda_2,\lambda_3, \ldots, \lambda_{\ell})$ is a non-increasing sequence of positive integers.  Each $\lambda_i$ of a partition \\ $\lambda=(\lambda_1,\lambda_2,\lambda_3, \ldots, \lambda_{\ell})$ is called a {\it{part of}} $\lambda$.  The sum of all the parts of $\lambda$ is denoted by $|\lambda|$ and the total number of parts of $\lambda$ is denoted by $\ell(\lambda)$.  We say that $\lambda$ is a partition of $n \in \mathbb{N}$ if $|\lambda| = n$, and we denote $\mathcal{P}^n$ as the set of all partitions of $n$.  The set of all partitions will be denoted as simply $\mathcal{P}$. 

A concave composition can now be written as a tuple $(\lambda^-,c,\lambda^+)$, where $\lambda^-$ and $\lambda^+$ are partitions (possibly empty) and where the smallest part of both $\lambda^-$ and $\lambda^+$ is strictly greater than the central part $c$. Let $X_k^\pm ((\lambda^-,\lambda^+))$ denote the frequencies of $(\lambda^-, \lambda^+)$. In other words, $X_k^+ ((\lambda^-,\lambda^+))$ is the number of parts of $\lambda^+$ that equal $k$ and $X_k^- ((\lambda^-,\lambda^+))$ is the number of parts of $\lambda^-$ that equal $k$. With this notation, \eqref{ccdef} can be rewritten as
\[
\sum_{k=1}^{\infty}kX_k^{+}+c+\sum_{k=1}^{\infty}kX_{k}^{-}=n.
\]

Concave compositions can also be represented graphically where each part is represented by a column of boxes.
\begin{ex}
	\label{ccshapeex}
	For $c = 1$, $\lambda^- = (4,4,3,2)$, and $\lambda^+ = (2,3,3)$, we see that $$(4,4,3,2, \underline{1},2,3,3)$$ is a concave composition of $n = 22$.  The graphical representation of this concave composition is
	\[
	(4,4,3,2, \underline 1,2,3,3) = \tableau[scY]{ & & \bl & \bl & \bl & \bl & \bl & \bl \cr & && \bl & \bl & \bl & & \cr & & & & \bl & & & \cr & & & & \tf & & & }
	\]
	where the bold box represents the central part $c = 1$.
\end{ex}
Let $V(n)$ be the number of concave compositions of $n$. For example, $V(3) = 13$ since all the concave compositions of $3$ are
\begin{align*}
\{&(\underline 0,3), (3, \underline 0), ( \underline 0,1,2), (2,1, \underline 0), ( \underline 0,1,1,1), (1,1,1, \underline 0), ( \underline 1,2), (2, \underline 1), (1, \underline 0,2),  \\
&(2, \underline 0,1), (1, \underline 0,1,1), (1,1, \underline 0,1), ( \underline 3)\}\,,
\end{align*}
where the central part $c$ of each concave composition is underlined.

Let $\mathbb{P}_n$ denote the uniform probability measure on all concave compositions of $n$. We are interested in certain statistics of concave compositions with respect to $\mathbb{P}_n$. The {\it length} of a concave composition is the total number of parts, $\ell(\lambda^+)+\ell(\lambda^-)+1$. The {\it tilt}  of a concave composition is the number of parts of $\lambda^+$ minus the number of parts of $\lambda^-$, $\ell(\lambda^+)-\ell(\lambda^-)$. The {\it half-perimeter} of a concave composition is the sum of the length plus the largest part of $\lambda^-$ and $\lambda^+$, i.e. $\max\{k: X_k^{+}\neq 0~\text{or}~X_k^{-}\neq 0\}$. 

Without loss of generality, we can assume that $c=0$ and consider concave compositions $(\lambda^-, \lambda^+) = (\lambda^-,c,\lambda^+)$. We can make this assumption about the central part $c$ by comparing Theorem 1.4 of \cite{MR3152010} with Theorem 6.2 of \cite{MR1634067}. By \cite[Theorem 1.4]{MR3152010}, the number of concave compositions of $n$ is,
$$
V(n) = \dfrac{\sqrt{6}}{(12n)^{\frac{5}{4}}} \hspace{0.01in} e^{\frac{\pi\sqrt{12n}}{3}} \left(1 + {O}\left( \dfrac{1}{\sqrt{n}} \right) \right)\,.
$$
In contrast, if $p_2(n)$ is the number of pairs of partitions $(\lambda^-,\lambda^+)$ with $|\lambda^-|+|\lambda^+|=n$, then by \cite[Theorem 6.2]{MR1634067},
\begin{equation}
\label{p2approx}
p_2(n) = \dfrac{\sqrt{6}}{(12n)^{\frac{5}{4}}} \hspace{0.01in} e^{\frac{\pi\sqrt{12n}}{3}} \left(1 + {O}\left( \dfrac{1}{\sqrt{n}} \right) \right)\,.
\end{equation}
Therefore, 
\begin{equation}\mathbb{P}_n(c = 0) = 1 + {O}(n^{-1/2}).\label{ciszero}\end{equation}

%------------------------------------------------------------------------------------------------
\section{The Boltzmann measure}
\label{BMsec}
%------------------------------------------------------------------------------------------------
In this section, we will introduce the Boltzmann measure which will be more convenient for our methods than the uniform measure $\mathbb{P}_n$. The measure will be established by applying Fristedt's conditioning device as it was employed in \cite{MR1667320,MR2422389,MR2915644}. Our goal in this section is to prove the Prokhorov distance between $\mathbb{P}_n$ and the Boltzmann measure converges to $0$ as $n\rightarrow\infty$ (Equation~\eqref{upperbound}). Our approach will closely follow \cite[Lemma 4.6]{MR1094553} although some of the proofs will resemble those in \cite{MR2422389}. 

For an arbitrary $n \in \mathbb{N}$ and $q \in (0,1)$ we define the Boltzmann distribution, say $Q_q$, on pairs of partitions $(\lambda^-,\lambda^+)$ as
\[Q_q((\lambda^-,\lambda^+)) = q^{|\lambda^-|+|\lambda^+|} \prod_{k=1}^{\infty} (1 - q^k)^2. 
\]

By Euler \cite{MR1634067},
$$
\sum_{\lambda\in \mathcal{P}}q^{|\lambda|} = (q;q)_\infty ^{-1}\,,
$$
where $(z;q)_n = \prod_{j=0}^{n-1} (1-zq^j)$ and $(z;q)_{\infty} = \prod_{j=0}^{\infty} (1 - zq^j)$.  Consequently,

$$
\sum_{n=0}^{\infty} p_2(n)q^n =(q;q)_\infty^{-2}\,.
$$
This gives us

\begin{equation}
\label{Measure1} 
\sum_{\substack{(\lambda^-, \lambda^+)\in \mathcal{P}\times\mathcal{P} \\|\lambda^-|+|\lambda^+|=n}} Q_q((\lambda^-,\lambda^+)) = p_2(n)q^n \prod_{k=1}^{\infty} (1 - q^k)^2\,, 
\end{equation}
\begin{equation} 
\label{Measure2} 
\sum_{n=1}^{\infty} \sum_{\substack{(\lambda^-, \lambda^+)\in \mathcal{P}\times\mathcal{P} \\|\lambda^-|+|\lambda^+|=n}} Q_q((\lambda^-,\lambda^+)) = 1\,. 
\end{equation}

Equations \eqref{Measure1} and \eqref{Measure2} tell us that we can view $Q_q$ as the probability measure for an experiment in which a concave composition is chosen at random and in which the integer $N := \sum_{k=1}^{\infty} (kX_k^++kX_k^-)$ being partitioned is itself random. 

The Boltzmann measure $Q_q$ decomposes further into a product of measures on the frequencies of $(\lambda^-,\lambda^+)$:
\begin{align*}
Q_q((\lambda^-, \lambda^+)) &= q^{|\lambda^-|+|\lambda^+|} \prod_{j=1}^{\infty} (1 - q^j)^2 \\
&= q^{\sum_{k=1}^{\infty} (kX_k^++kX_k^-)} \prod_{j=1}^{\infty} (1 - q^j)^2 \\
&= \prod_{k=1}^{\infty}q^{kX_k^+}(1-q^k) \prod_{j=1}^{\infty}q^{jX_j^-}(1-q^j)\,,
\end{align*}
where we can identify the frequencies of $(\lambda^-,\lambda^+)$ as independent geometric random variables.  We recover $\mathbb{P}_n$ by conditioning that $|\lambda^-| + |\lambda^+| = n$. In other words, 
\begin{align}\label{conditioning}
\mathbb{P}_n(S)=Q_q\left(S \big\vert~N= n\right),~\forall S\in\mathcal{P}\times\mathcal{P}.
\end{align} 
This motivates us to set $q$ such that {\it{most}} of the probability is centered around a fixed integer $n$.  Thus we aim to choose a sequence $q = q_n$ such that
\begin{equation}
\label{expectedval}
\mathbb{E}_{q_n}(|\lambda^+|+|\lambda^-|) \approx n\,.
\end{equation}
Such a sequence that could be the leading term approximation to the solution of Equation \eqref{expectedval} is
\begin{equation}
\label{qn}
q_n = e^{-\pi/\sqrt{3n}}\,.
\end{equation}

The following are some properties of the random variable $N$ under the probability distribution $Q_{q_{n}}$.
\begin{prop}\label{expvar}
	The expectation and variance of $N$ under $Q_{q_n}$ is given by
	$$
	\mu_n(N) = \sum_{k=1}^{\infty}\frac{2kq_n^k}{1-q_n^k} \hspace{0.2in} \text{ and } \hspace{0.2in}
	\sigma^2_n(N) = \sum_{k=1}^{\infty}\frac{2k^2q_n^k}{(1-q_n^k)^2}\,,
	$$
	respectively. In addition, if $\phi_n(t) = \mathbb{E}_{q_n}(e^{itN})$ then
	\[
	\phi_n(t)=\prod\limits_{k=1}^{\infty}\left(\frac{1-q_n^k}{1-q_n^ke^{itk}}\right)^2.
	\]
\end{prop} 
\begin{proof}
	
	The expectation can be found by summing over the expectations of the random variables $X_k^+$ and $X_k^-$,
	\begin{align*}
	\mu_n(N)=\sum_{k=1}^{\infty}\left(k\cdot\mathbb{E}(X_k^+)+k\cdot\mathbb{E}(X_k^-)\right)=\sum_{k=1}^{\infty}\frac{2kq_n^k}{1-q_n^k}.
	\end{align*}
	The variance can be computed similarly as
	\begin{align*}
	\sigma^2_n(N)=\sum_{k=1}^{\infty}\left(k^2\cdot\mathrm{Var}(X_k^+)+k^2\cdot\mathrm{Var}(X_k^-)\right)=\sum_{k=1}^{\infty}\frac{2k^2q_n^k}{(1-q_n^k)^2}.
	\end{align*}
	By definition,
	\begin{align*}
	\phi_n(t)&=\sum_{\substack{(\lambda^-, \lambda^+)\in \mathcal{P}\times\mathcal{P} \\|\lambda^-|+|\lambda^+|=n}} e^{itN}Q_{q_n}((\lambda^-,\lambda^+))\\
	&= \sum_{\substack{(\lambda^-, \lambda^+)\in \mathcal{P}\times\mathcal{P} \\|\lambda^-|+|\lambda^+|=n}} \left(q_ne^{it}\right)^{N} \prod_{k=1}^{\infty} (1 - q_n^k)^2.\\
	\end{align*}
	By Euler (see \cite{MR1634067}),
	\begin{align*}
	\phi_n(t)&=\prod_{k=1}^{\infty}\left(\frac{1 - q_n^k}{1-q_n^ke^{itk}}\right)^2.\\
	\end{align*}
	%The result follows by the fact that
	%\[
	%\varphi_n(t)=e^{\frac{-it\mu_n}{\sigma_n}}\phi_n\left(\frac{t}{\sigma_n}\right).
	%\]
\end{proof} 

\begin{corr}\label{asymvarexp}
	As $n\rightarrow \infty$,
	$$
	n-\mu_n(N)=o(n^{3/4}) \hspace{0.2in} \text{ and } \hspace{0.2in}
	\sigma^2_n(N)={\rm \Theta}\left(\frac{\sqrt{12}n^{3/2}}{\pi}\right)\label{sd}.
	$$
\end{corr}
\begin{proof}

	By the proof of Corollary 4.4 of \cite{MR1094553}, $(1/2)\mu_n(N)$ is asymptotic to $\frac{\pi^2}{6\ln^2(1/q_n)}$ with an error of $\frac{1}{\ln(1/q_n)}$, and $(1/2)\sigma^2_n(N)$ is asymptotic to $\frac{\pi^2}{3\ln^3(1/q_n)}$. Plugging in $q_n=\exp\left(\frac{-\pi}{\sqrt{3n}}\right)$ and multiplying by $2$ gives the results.
\end{proof}

Now let $K^{\pm}_n$ be any two sets of positive integers such that 
\begin{equation}\label{smalloverK}
\sum_{k\in K^{\pm}_n}\frac{k^2q^k_n}{(1-q_n^k)^2}=o(n^{3/2}),
\end{equation}
and say $d^{\pm}_n$ are the cardinalities of the sets $K^{\pm}_n$. Define
\begin{align}
&W_n:\{(\lambda^-, \lambda^+)\in \mathcal{P}\times \mathcal{P}\colon |\lambda^-|+|\lambda^+|=n\}\rightarrow\mathbb{R}^{d_n^{+}} \times \mathbb{R}^{d_n^{-}}\nonumber\\
&W_n:(\lambda^{-}, \lambda^{+})\mapsto(X^{+}_{k_1}(\lambda^{-}, \lambda^{+}), X^{-}_{k_2}(\lambda^{-}, \lambda^{+}):~k_1\in K^{+}_n, k_2\in K^{-}_n)\label{defWn}
\end{align}
and let
\begin{align}\label{defBn}
B_n= \Bigg\{w_n=(x_{k_1,n}, y_{k_2,n}:~k_1\in K^{+}_n, k_2\in K^{-}_n):&\\
&\hspace{-2.3in}\left\vert\sum\limits_{k\in K^+_n}k x_{k,n}+\sum\limits_{k\in K^-_n}ky_{k,n}-\sum\limits_{k\in K^+_n}\frac{kq_n^k}{1-q_n^k}-\sum\limits_{k\in K^-_n}\frac{kq_n^k}{1-q_n^k}\right\vert\leq a_n\Bigg\}, \nonumber
\end{align}
where $\{a_n\}_{n=1}^{\infty}$ is such that $a_n=o(n^{3/4})$. This gives us a lemma analogous to \cite[Lemma 4.2]{MR1094553}.

\begin{lemma}\label{UBLem} For all $w_n\in B_n$, if
	\begin{align}\label{GTONE}
	\frac{Q_{q_n}\left(N=n\big\vert~W_n=w_n\right)}{Q_{q_n}\left(N=n\right)}\rightarrow 1
	\end{align}
	uniformly as $n\rightarrow\infty$, then for all Borel sets $B\subseteq \mathbb{R}^{d^{+}_n}\times \mathbb{R}^{d^{-}_n}$,
	\begin{align}\label{upperbound}
	\sup_{B} |\mathbb{P}_n(W_n^{-1}(B)) - Q_{q_n}(W_n^{-1}(B))|\rightarrow0.
	\end{align}
\end{lemma}
\begin{proof}
	Combining~\eqref{conditioning} and the fact that
	\begin{align}
	Q_{q_n}\left(S \big\vert~N= n\right)=\frac{Q_{q_n}\left(S\cap N=n\right)}{Q_{q_n}\left(N=n\right)},
	\end{align}
	the left--hand--side of~\eqref{upperbound} is equivalent to
	\begin{align}
	\sup_{B} \left|\frac{Q_{q_n}\left(W_n^{-1}(B)\cap N=n\right)}{Q_{q_n}\left(N=n\right)}- Q_{q_n}(W_n^{-1}(B))\right|.
	\end{align}
	The quantity in the absolute value is bounded above by
	\begin{align*}
	Q_{q_n}(W_n^{-1}(B\cap B^{c}_n))&\\
	&\hspace{-2.5cm}+\sum_{w_n\in B\cap B_n}\left(Q_{q_n}(W_n^{-1}(w_n))-\frac{Q_{q_n}\left(W_n^{-1}(w_n)\cap N=n\right)}{Q_{q_n}\left(N=n\right)}\right)\\
	\leq Q_{q_n}(W_n^{-1}(B\cap B^{c}_n))&\\
	&\hspace{-2.5cm}+\sum_{w_n\in B\cap B_n}Q_{q_n}(W_n^{-1}(w_n))\left|1-\frac{Q_{q_n}\left(N=n\big\vert~W_n=w_n\right)}{Q_{q_n}\left(N=n\right)}\right|.
	\end{align*}
	By~\eqref{GTONE}, the quantity in the absolute value goes to $0$. In addition, by the definition of $B_n$ and Chebyshev's Inequality,
	\begin{align*}
	Q_{q_n}(W_n^{-1}(\mathbb{R}^{d^{+}_n}\times \mathbb{R}^{d^{-}_n}\cap B^c_n))&\\
	&\hspace{-4.5cm}=Q_{q_n}\left(w_n:~\left\vert\sum\limits_{k\in K^+_n}k x_{k,n}+\sum\limits_{k\in K^-_n}ky_{k,n}-\sum\limits_{k\in K^+_n}\frac{kq_n^k}{1-q_n^k}-\sum\limits_{k\in K^-_n}\frac{kq_n^k}{1-q_n^k}\right\vert>a_n\right)\\
	&\hspace{-4.5cm}\leq\frac{1}{a_n^2}\left(\sum_{k\in K^+_n}\frac{k^2q_n^k}{(1-q_n^k)^2}+\sum_{k\in K^-_n}\frac{k^2q_n^k}{(1-q_n^k)^2}\right)
	\end{align*}
	which approaches $0$ by~\eqref{smalloverK}. Therefore, if condition~\eqref{GTONE} holds, the Prokhorov distance between $\mathbb{P}_n$ and $Q_{q_n}$ (Equation~\eqref{upperbound}) vanishes as $n\rightarrow\infty$.
\end{proof}

After applying Lemma~\ref{UBLem}, proving that the Prokhorov distance between $\mathbb{P}_n$ and $Q_{q_n}$ converges to $0$ reduces to showing that~\eqref{GTONE} holds for $w_n\in B_n$.
To do so, we will show that the numerator and the denominator of~\eqref{GTONE} are asymptotically equivalent. First, we asymptotically compute the denominator. As in \cite{MR1094553}, we will show that the distribution of $N$ under $Q_{q_n}$ can be approximated by the normal distribution; however, our proof will rely on the Lyapunov condition as was done in \cite{MR2422389}. 

\begin{lemma}\label{aprxnorm}
	Under $Q_{q_n}$, as $n\rightarrow\infty$,
	\[
	\frac{N-\mu_n(N)}{\sigma_n(N)}\stackrel{d}{\rightarrow}N(0,1)\,.
	\]
\end{lemma}
\begin{proof}
	
	%For each $n\in\mathbb{N}$, consider the triangular array $\left\{X^{+}_{n,k}+X^{-}_{n,k}:~n,k\geq 1\right\}$ where $X^{+}_{n,k}$ and $X^{-}_{n,k}$ are i.i.d. geometric random variables under the probability distribution $Q_{q_n}$ with parameter $1-q_n^k$. 
	%

	The statement will follow after verifying the Lyapunov condition for $\delta=1$ (see for example, \cite{MR0203748}). More precisely, we will verify that as $n \rightarrow \infty$,
	\[
	\frac{1}{\sigma_n^{3}(N)}\sum\limits_{k=1}^{n}\mathbb{E}\left(|Y_{n,k}|^{3}\right)\rightarrow 0\,,
	\]
	where $Y_{n,k}=\left(kX^{+}_{k}+kX_k^{-}\right)-\mathbb{E}\left(kX^{+}_{k}+kX_k^{-}\right)$.
	
	First note that
	\begin{align}
	\mathbb{E}\left(|Y_{n,k}|^3\right)
	&\leq4k^3\left(\mathbb{E}\left(X^{+}_{k}+X_k^{-}\right)^{3}+\mathbb{E}\left(X^{+}_{k}+X_k^{-}\right)^{3}\right)\nonumber\\
	&\leq 8k^3\mathbb{E}\left(X^{+}_{k}+X_k^{-}\right)^{3}\leq 32k^3\left(\mathbb{E}\left(X^{+}_{k}\right)^3+\mathbb{E}\left(X^{-}_{k}\right)^3\right)\label{expineq},
	\end{align}
	where the first and last inequalities are due to the $c_r$-inequality (see for example, \cite{MR0203748}).  Therefore, we need only consider an upper bound for 
	\begin{equation}\label{ineqt}
	\sum\limits_{k=1}^{n}k^3\left(\mathbb{E}\left(X^{+}_{k}\right)^3+\mathbb{E}\left(X^{-}_{k}\right)^3\right)\leq\sum\limits_{k=1}^{n}\frac{2k^3q_n^k}{(1-q^k_n)^3}\,,
	\end{equation}
	since $X^{+}_{k}$ and $X^{-}_{k}$ are i.i.d. random variables.  Now by the Euler-Maclaurin Formula, 
	\begin{equation}\label{afteru}
	\sum\limits_{k=1}^{n}\frac{2k^3q_n^k}{(1-q^k_n)^3}\sim\int\limits_{1}^{\infty}\frac{2t^3q_n^t}{(1-q_n^{t})^3} dt=\frac2{\ln^4(q_n^{-1})}\int_{1}^{\infty}\frac{u^3e^{-u}}{(1-e^{-u})^3} du
	\end{equation}
	after the substitution $u=t\ln(q_n^{-1})$.  Notice that the integral on the right-hand-side of \eqref{afteru} is finite.  Therefore, substituting $q_n=e^{\frac{-\pi}{\sqrt{3n}}}$ into~\eqref{afteru},
	%\begin{equation}\label{evalint}
	%\int\limits_{0}^{\infty}\frac{u^3e^{-u}}{(1-e^{-u})^3}dt=3\zeta(3)+\frac12\pi^2.
	%\end{equation}
	%Therefore,  and combining with \eqref{evalint} we obtain
	\begin{equation}\label{asymptonthirdmom}
	\sum\limits_{k=1}^{n}\frac{2k^3q_n^k}{(1-q_n^k)^3}=O\left(n^2\right)=o\left(n^{9/4}\right).
	\end{equation}
	By Corollary~\ref{asymvarexp},
	\[
	\left(\sigma_n^2\left(N\right)\right)^{3/2}={\rm \Theta}(n^{9/4})\,, 
	\]which completes the proof.
\end{proof}

We will strengthen Lemma~\ref{aprxnorm} to the local limit theorem at $0$ which will give the desired approximation. But first, we will prove the following lemma to help with the calculation.

\begin{lemma}\label{termCountingLemma}
    Fix $\pi \sigma_n^{1/3}(N) < |t| < \pi \sigma_n(N)$ and let 	\[
	S:=\left\{k: \left\lceil\frac{1}{2} \sigma_n^{2/3}(N)\right\rceil \le k \leq \left\lfloor \sigma_n^{2/3}(N)\right\rfloor,~\cos(kt/\sigma_n(N))\leq 0\right\}.
	\]
	For sufficiently large $n$, $|S|\geq \frac{1}{3}\sigma_n^\frac23(N)
	$.
    
\end{lemma}
\begin{proof}
Let $\omega_n:=\frac{t}{2\pi \sigma_n(N)}$. Notice that $\cos(2\pi k\omega_n)\leq 0$ if and only if $\{k\omega_n\} \in \left[\frac14,\frac34\right]$ where $\{x\}$ denotes the fractional part of $x$, i.e. $\{x\}:=x-\lfloor x\rfloor$. Therefore, 
$$
  \lim_{n\to\infty}\frac{1}{\sigma_n^\frac{2}{3}(N)}|S| =
 \lim_{n\to\infty}\frac{1}{\sigma_n^\frac{2}{3}(N)}  \sum_{k=\left\lceil\frac{1}{2} \sigma_n^{2/3}(N)\right\rceil}^{\left\lfloor\sigma_n^\frac23(N)\right\rfloor} {\bf 1}_{[\frac14,\frac34]}\left(\left\{\omega_n k\right\}\right). 
 $$

Now consider $\lambda := |t|/2\pi\sigma_n^{1/3}(N)$. The equation above can be rewritten in terms of $\lambda$ and this Riemann sum converges,
 \begin{align*}
 \lim_{n\to\infty}\frac{1}{\sigma_n^\frac{2}{3}(N)} \sum_{k=\left\lceil\frac{1}{2} \sigma_n^{2/3}(N)\right\rceil}^{\left\lfloor\sigma_n^\frac23(N)\right\rfloor} {\bf 1}_{[\frac14,\frac34]}\left(\left\{\lambda \frac{k}{\sigma_n^\frac23(N)}\right\}\right) &= \int_{\frac{1}{2}}^1 {\bf 1}_{[\frac14,\frac34]}\left(\left\{\lambda x\right\}\right)dx 
 \end{align*}

Since $\sigma_n^{1/3}(N)\pi <  |t|,$ we observe $\lambda>\frac{1}{2}$, and therefore,
\begin{align*}
     \int_{\frac12}^1 {\bf 1}_{[\frac14,\frac34]}\left(\left\{\lambda x\right\}\right)dx&\geq \int_{\frac12}^1 {\bf 1}_{[\frac14,\frac34]}\left(\left\{\frac{1}{2} x\right\}\right)dx =\frac12.
\end{align*}

%$|t|<\pi\omega_n(N)$,  $|\omega_n|<\frac{1}{2}$ and therefore,

%If $|\omega_n| := \frac{|t|}{2\pi \sigma_n(N)} :=  \frac{\lambda}{\sigma_n^\frac{2}{3}(N)}<\frac{1}{2}$ then  Therefore,
 
% By $\sigma_n^\frac{1}{3}(N)\pi <  |t|,$ we observe $\lambda>\frac{1}{2}.$  For the moment, consider any $\lambda>\frac{1}{2}$ and the following Riemann sum converges

 Therefore, $|S|\geq\frac13\sigma_n^{2/3}(N)$ for sufficiently large $n$.
\end{proof}
\begin{lemma}\label{locallimitthm}
	As $n\rightarrow\infty$,
	\[
	Q_{q_n}(N = n)\sim\frac{1}{\sqrt[4]{48n^{3}}}\,.
	\]
\end{lemma}
\begin{proof}
Let $\varphi_n(t)$ be the characteristic function of $\frac{N-\mu_n(N)}{\sigma_n(N)}$.	By~\cite[Theorem 2.9]{MR1235434}, a local limit theorem holds if there exists an integrable function $f$ such that for each $t \in \mathbb{R}$,
	\begin{equation}\label{indfun}
	\sup\limits_n |\varphi_n(t)|{\bf{1}}_{\{|t|\leq\pi\sigma_n^{1/3}(N)\}}(t)\leq f(t)
	\end{equation}
	and
	\begin{equation}\label{littleosigma}
	\sup\limits_{\pi \sigma_n^{1/3}(N)\leq|t|\leq\pi\sigma_n(N)} |\varphi_n(t)|=o\left(\frac{1}{\sigma_n(N)}\right).
	\end{equation}

	To prove~\eqref{indfun} and~\eqref{littleosigma}, we will first establish an upper bound on $\phi_n(t)$, the characteristic function of $N$, and use that to obtain an upper bound for the characteristic function of $\frac{N-\mu_n(N)}{\sigma_n(N)}$.
	
	By Proposition~\ref{expvar},
	\[
	|\phi_n(t)|=\exp\left(-2\sum\limits_{k=1}^{\infty}\ln\Bigg|\frac{1-q_n^ke^{itk}}{1-q_n^k}\Bigg|\right).
	\]
	We will obtain an upper bound on this expression by making the sum smaller. To do so, notice that
	\begin{align}
	\Bigg\lvert\frac{1-q_n^ke^{itk}}{1 - q_n^k}\Bigg\rvert=\sqrt{1+\frac{2q_n^k(1-\cos(kt))}{(1-q_n^k)^2}}\geq \sqrt{1+2q_n^k(1-\cos(kt))}\label{boundonabsval}.
	\end{align}
	To continue, we need to establish two elementary facts. The first is that for $x>0$, $\ln(1+x)>\frac{x}{1+x}$. This can be derived by minimizing $\frac{(x+1)\ln(x+1)}{x}$ over the interval $x>0$. The second is that for $q_n<1,$  $$1+2q_n^k(1-\cos(kt) ) \le 1+ 2(1-\cos(kt))\le 5.$$  Together they combine to provide the following lower bound,
	$$
	\ln \left(1+2q_n^k(1-\cos(kt))\right) \geq \frac{2q_n^k(1-\cos(kt))}{1+ 2q_n^k(1-\cos(kt))}\geq \frac{2q_n^k(1-\cos(kt))}{5}.
	$$
	By \eqref{boundonabsval} and the lower bound above, we obtain
	\[
	|\phi_n(t)|\leq\exp\left(-\sum\limits_{k=1}^{\infty}\frac{2q_n^k(1-\cos(kt))}{ 5}\right).
	\]
	Therefore, we have
	\begin{equation}\label{boundonchar}
	\big\lvert \varphi_n(t)\big\rvert\leq\exp\left(-\frac25\sum\limits_{k=1}^{\infty}q_n^k(1-\cos(kt/\sigma_n(N)))\right).
	\end{equation}
	
	To prove~\eqref{indfun}, restrict $t$ such that $|t|\leq\pi\sigma_n^{1/3}(N)$ and restrict the sum in~\eqref{boundonchar} to values $\left\lceil \frac12\sigma_n^{2/3}(N)\right\rceil\leq k \leq \left\lfloor \sigma_n^{2/3}(N)\right\rfloor$. Therefore, $kt/\sigma_n(N)$ is strictly bounded above by $\pi$ and by optimizing $\frac{1- \cos(x)}{x^2}$ for $|x|<\pi$ we get
	
	$$1-\cos(x)\geq\frac{1}{\pi^2}\left(1-\cos(\pi)\right)x^2=\frac{2x^2}{\pi^2}.$$
	Thus,
	\begin{equation}
	|\varphi_n(t)|\leq\exp\left(-\frac{4}{5\pi^2}\sum\limits_{\left\lceil \frac12\sigma_n^{2/3}(N)\right\rceil}^{\left\lfloor \sigma_n^{2/3}(N)\right\rfloor}\frac{q_n^k(kt)^2}{\sigma_n^2(N)}\right). \label{boundforfirstcondition}
	\end{equation}
	For $k\geq\left\lceil \frac12\sigma_n^{2/3}(N)\right\rceil$, we have that $k={\rm \Theta}(\frac{\sqrt[6]{12}n^{1/2}}{2\pi^{1/3}})$ by Corollary~\ref{asymvarexp}. Thus, for 
	$q_n^k=e^{-k\pi/\sqrt{3n}}$ , we can establish the upper bound 
	\[q_n^k\geq e^{-\sqrt[3]{\pi^2/12}}\geq c>0\] for some absolute constant $c$. Since $|t|\leq\pi \sigma_n^{1/3}(N)$ and the sum is over $\frac{1}{2}\sigma_n^{2/3}(N)$ terms, we finally prove the required upper bound
	\[
	\big\lvert\varphi_n(t)\big\rvert\leq\exp\left(-\frac{c}{10 \pi^2}  t^2\right)
	\] which proves~\eqref{indfun}.
	
	To prove~\eqref{littleosigma}, restrict the sum in~\eqref{boundonchar} to the set defined in Lemma~\ref{termCountingLemma}, 
\[
	S=\left\{k: \left\lceil\frac{1}{2} \sigma_n^{2/3}(N)\right\rceil \le k \leq \left\lfloor \sigma_n^{2/3}(N)\right\rfloor,~\cos(kt/\sigma_n(N))\leq 0\right\}.
	\]
	By Lemma \ref{termCountingLemma}, 
	$
	|S| \geq \frac{1}{3}\sigma_n^\frac23(N) 
	$
	for sufficiently large $n.$ As shown above, $q_n^k\geq c>0$ for some absolute constant $c$, and thus,
	\[
	|\varphi_n(t)|\leq\exp\left(-\frac {2}{15}c\sigma_n^{2/3}(N)\right)\,,
	\]
	for sufficiently large $n$.  Therefore,~\eqref{littleosigma} holds.
	
	Since a local limit theorem holds, then
	\[
	Q_{q_n}(N=n)=\frac{1}{\sqrt{2\pi}}\frac{1}{\sigma_n(N)}\approx\frac{1}{\sqrt[4]{48n^3}}\,,
	\]
	as desired.
\end{proof}

At this point, we will return to \eqref{GTONE} and asympototically compute the numerator. First notice that
%\mandy{Avi is figuring out how to align this properly}
\begin{align}
\label{algadj}
Q_{q_n}\left(N=n\big\vert~W_n=w_n\right)& = \\
&\hspace{-0.8in}Q_{q_n}\left(\sum\limits_{k\notin K_n^{+}}kX_k^{+}+\sum\limits_{k\notin K^{-}_n}kX^{-}_k= n-\sum\limits_{k\in K^+_n}kx_{k,n}-\sum\limits_{k\in K^-_n}ky_{k,n}\right) \nonumber.
\end{align}
Now, we will consider a variation of $N$ defined as, 
\[\widehat N:=\sum\limits_{k\notin K_n^{+}}kX_k^{+}+\sum\limits_{k\notin K^{-}_n}kX^{-}_k.
\] As was done with the denominator of~\eqref{GTONE}, we will show that the distribution of $\widehat N$ can be approximated by the normal distribution. We begin by computing the expectation and variance of $\widehat N$.

\begin{lemma}\label{asymvarexphat}
	If $q=q_n$, then as $n\rightarrow\infty$\,,
	$$
	n-\mu_n(\widehat N)=o(n^{3/4}) \hspace{0.2in} \text{ and } \hspace{0.2in} \sigma^2_n(\widehat N)={\rm \Theta}\left(\frac{\sqrt{12}n^{3/2}}{\pi}\right).
	$$
\end{lemma}
\begin{proof}
	First notice that
	\[
	\mu_n(\widehat N)=\sum\limits_{k\notin K_n^+}\frac{kq_n^k}{1-q_n^k} + \sum\limits_{k\notin K_n^-}\frac{kq_n^k}{1-q_n^k}.
	\]
	Therefore,
	%\mandy{Avi is figuring out how to align this properly}
	\begin{align*}
	\left(n-\sum\limits_{k\in K^+_n}kx_{k,n}-\sum\limits_{k\in K^-_n}ky_{k,n}\right) - & \mu_n(\widehat N) = \\
	&\hspace{-2in} n-\mu_n(N)+\left(\sum_{k\in K^+_n}\frac{kq_n^k}{1-q_n^k}+\sum_{k\in K^-_n}\frac{kq_n^k}{1-q_n^k}-\sum\limits_{k\in K^+_n}kx_{k,n}-\sum\limits_{k\in K^-_n}ky_{k,n}\right).
	\end{align*}
	By Corollary~\ref{asymvarexp}, $n-\mu_n(N)=o(n^{3/4})$. In addition, the definition of $B_n$ says that the difference on the right is also $o(n^{3/4})$. Therefore, the first result follows.
	
	Now by independence,
	\begin{align*}
	\sigma^2_n(\widehat N)&=\sum\limits_{k\notin K^+_n}\frac{k^2q_n^k}{(1-q_n^k)^2}+\sum\limits_{k\notin K^-_n}\frac{k^2q_n^k}{(1-q_n^k)^2}\\
	&\hspace{-.5in}=\sum\limits_{k=1}^{\infty}\frac{2k^2q_n^k}{(1-q_n^k)^2}-\sum\limits_{k\in K^+_n}\frac{k^2q_n^k}{(1-q_n^k)^2}-\sum\limits_{k\in K^-_n}\frac{k^2q_n^k}{(1-q_n^k)^2}\,.
	\end{align*}
	By Corollary~\ref{asymvarexp}, the first sum is ${\rm \Theta}\left(\frac{\sqrt{12}n^{3/2}}{\pi}\right)$, and by Equation~\eqref{smalloverK} the last two sums are $o(n^{3/2})$. Therefore, the second result follows.
\end{proof}

\begin{lemma}\label{aprxnormhat}
	Under $Q_{q_n}$, as $n\rightarrow\infty$,
	\[
	\frac{\widehat N-\mu_n(\widehat N)}{\sigma_n(\widehat N)}\stackrel{d}{\rightarrow}N(0,1)\,.
	\]
\end{lemma}
\begin{proof}
	As in Lemma~\ref{aprxnorm}, we will prove this statement by verifying the Lyapunov condition. Analogous to~\eqref{expineq} and~\eqref{ineqt}, we establish the following bound using Equation~\eqref{asymptonthirdmom}, 
	
	\begin{align*}
	\sum\limits_{k\notin K^{+}_n}\mathbb{E}\left(kX_k^{+}\right)^3+\sum\limits_{k\notin K^{-}_n}\mathbb{E}\left(kX_k^{-}\right)^3&\\
	&\hspace{-.75in} =\displaystyle\sum\limits_{k=1}^{\infty}\frac{2k^3q_n^k}{(1-q_n^k)^3}-\sum\limits_{k\in K^{+}_n}\frac{k^3q_n^k}{(1-q_n^k)^3}-\sum\limits_{k\in K^{-}_n}\frac{k^3q_n^k}{(1-q_n^k)^3}\\
	&\hspace{-.75in}\leq \displaystyle\sum\limits_{k=1}^{\infty}\frac{2k^3q_n^k}{(1-q_n^k)^3}=o\left(n^{9/4}\right).
	\end{align*}
	%
	%\begin{eqnarray*}
	%\sum\limits_{(k_1,k_2)\notin K^{+}\times K^{-}}\mathbb{E}\left(k_1X^{+}_{n,k_1}+k_2X^{-}_{n,k_2}\right)^{3}&\\
	%&\hspace{-2.85in}=\sum\limits_{k_1\notin K^{+}_n}\frac{k_1^3q_n^{k_1}}{\left(1-q_n^{k_1}\right)^3}+\sum\limits_{(k_1,k_2)\notin K^{+}\times K^{-}}\frac{3k_1^2k_2q_n^{k_1}q_n^{k_2}}{\left(1-q_n^{k_1}\right)^2\left(1-q_n^{k_2}\right)}\\
	%&\hspace{-2in}+\sum\limits_{(k_1,k_2)\notin K^{+}\times K^{-}}\frac{3k_1k_2^2q_n^{k_1}q_n^{k_2}}{\left(1-q_n^{k_1}\right)\left(1-q_n^{k_2}\right)^2}+\sum\limits_{k_2\notin K^{-}_n}\frac{k_2^3q_n^{k_2}}{(1-q_n^{k_2})^3}\\ \\ \\
	%&\hspace{-3in}=\sum\limits_{k=1}^{\infty}\frac{8k^3q_n^k}{(1-q_n^k)^3}-\sum\limits_{k_1\in K^{+}_n}\frac{k_1^3q_n^{k_1}}{\left(1-q_n^{k_1}\right)^3}-\sum\limits_{k_2\in K^{-}_n}\frac{k_2^3q_n^{k_2}}{(1-q_n^{k_2})^3}\\
	%&\hspace{-1.6in}-\sum\limits_{(k_1,k_2)\in K^{+}\times K^{-}}\frac{3k_1^2k_2q_n^{k_1}q_n^{k_2}}{\left(1-q_n^{k_1}\right)^2\left(1-q_n^{k_2}\right)}-\sum\limits_{(k_1,k_2)\in K^{+}\times K^{-}}\frac{3k_1k_2^2q_n^{k_1}q_n^{k_2}}{\left(1-q_n^{k_1}\right)\left(1-q_n^{k_2}\right)^2}\\ \\ \\
	%&\hspace{-4.48in}\leq \sum\limits_{k=1}^{\infty}\frac{8k^3q_n^k}{(1-q_n^k)^3}=o\left(n^{9/4}\right)
	%\end{eqnarray*}
	Now by Corollary~\ref{asymvarexp} and Equation~\eqref{smalloverK},
	\begin{align*}
	\left(\sigma_n^2(\widehat N)\right)^{3/2}&=\left(\sigma_n^2(N)-\sum\limits_{k\in K^{+}_n}\left(\frac{k^2q_n^k}{1-q_n^k}\right)-\sum\limits_{k\in K^{-}_n}\left(\frac{k^2q_n^k}{1-q_n^k}\right)\right)^{3/2}\\
	&=\left(\sigma_n^2(N)+o(n^{3/2})\right)^{3/2}={\rm \Theta}(n^{9/4}).
	\end{align*}
	Together, these calculations complete the proof.
	
\end{proof}

We also have the following Corollary implied by the proof of Lemma~\ref{asymvarexphat}.
\begin{corr} \label{aprxnormhat2}As $n\rightarrow\infty$,
	$$
	\lim_{n\to \infty}\frac{1}{\sigma_n(\widehat{N})}\left(\left(n-\sum\limits_{k\in K^+_n}kx_{k,n}-\sum\limits_{k\in K^-_n}ky_{k,n}\right)- \mu_n(\widehat{N})\right) =0.
	$$
\end{corr}

Now we strengthen Lemma~\ref{aprxnormhat} to the local limit theorem at $0$ which will give the desired approximation.
\begin{lemma}\label{locallimitthm2}
As $n\rightarrow\infty$,
	\[
	Q_{q_n}\left(\widehat N = n-\sum\limits_{k\in K^+_n}kx_{k,n}-\sum\limits_{k\in K^-_n}ky_{k,n}
\right)\sim\frac{1}{\sqrt[4]{48n^{3}}}\,.
	\]
\end{lemma}
\begin{proof}
	 As a consequence of Corollary~\ref{aprxnormhat2}, we need only show a local limit theorem exists to complete the proof. This work is analogous to the proof of Lemma~\ref{locallimitthm}.	
	
	Let $\widehat \varphi_n(t)$ denote the characteristic function of $\frac{\widehat N-\mu_n(\widehat N)}{\sigma_n({\widehat N})}$.   By~\cite[Theorem 2.9]{MR1235434} we need only show that
	\begin{equation}\label{indfun2}
	\sup\limits_n |\widehat \varphi_n(t)|{\bf{1}}_{\{|t|\leq\pi\sigma_n^{1/3}(\widehat N)\}}(t)\leq e^{-Ct^2}
	\end{equation}
	for some absolute constant $C$ and that
	\begin{equation}\label{littleosigma2}
	\sup\limits_{\pi\sigma_n^{1/3}(\widehat N)\leq|t|\leq\pi\sigma_n(\widehat N)} |\widehat \varphi_n(t)|=o\left(\frac{1}{\sigma_n(\widehat N)}\right).
	\end{equation}
	Equations~(\ref{indfun2}) and~(\ref{littleosigma2}) follow almost exactly as in the proof of Lemma \ref{locallimitthm}, but we will highlight a few modest adaptations that need to occur.
	
	The characteristic function of $\widehat{N}$, say $\widehat{\phi}_n(t)$, is similar to the characterstic function of $N$, $\phi_n(t)$, but with fewer nonzero terms. In particular,
	
	\begin{align*}
	\widehat \phi_n(t)&=\mathbb{E}\left[\exp\left(i t\left(\sum_{k\not \in K_n^{+}}k X_k^{+}+\sum_{k\not \in K_n^{-}}k X_k^{-}\right)\right)\right] \\ 
    &=\mathbb{E}\left[\prod_{k\not \in K_n^{+}}\exp\left(it k X_k^{+}\right)
    \prod_{k\not \in K_n^{-}}\exp\left(it k X_k^{-}\right)\right] \\
	&=\prod_{k\not \in K_n^{+}}\left(\frac{1-q_n^k}{1-q_n^k e^{itk}}\right)\prod_{k\not \in K_n^{-}}\left(\frac{1-q_n^k}{1-q_n^k e^{itk}}\right)
	\end{align*}
	Therefore, 
	\[
	|\widehat{\phi}_n(t)| = \exp \left(-\displaystyle\sum_{k=1}^{\infty}\ln \left|\displaystyle\frac{1-q_n^ke^{itk}}{1-q_n^k}\right|\left({\bf 1}_{k\not\in K_n^{+}}(k)+{\bf 1}_{k\not\in K_n^{-}}(k)\right)\right).
	\]
	From here, we can proceed exactly as in Lemma~\ref{locallimitthm} to establish the following bound similar to (\ref{boundforfirstcondition}),

		\[
		|\widehat \varphi_n(t)|\leq\exp\left(-\frac{2}{5\pi^2} \sum\limits_{\left\lceil \frac12\sigma_n^{2/3}(\widehat{N})\right\rceil}^{\left\lfloor \sigma_n^{2/3}(\widehat{N})\right\rfloor}\frac{q_n^k(kt)^2}{\sigma_n^2(\widehat{N})}\left({\bf1}_{k\not\in K_n^{+}}(k)+{\bf1}_{k\not\in K_n^{-}}(k)\right)\right).
	\]
At this point, we can write 

	\begin{align*}\sum\limits_{\left\lceil \frac12\sigma_n^{2/3}(\widehat{N})\right\rceil}^{\left\lfloor \sigma_n^{2/3}(\widehat{N})\right\rfloor}
	\frac{q_n^k(kt)^2}{\sigma_n^2(\widehat{N})}\left({\bf 1}_{k\not\in K_n^{+}}(k)+{\bf 1}_{k\not\in K_n^{-}}(k)\right)&\\
	&\hspace{-3cm}
	\geq \sum\limits_{\left\lceil \frac12\sigma_n^{2/3}(\widehat{N})\right\rceil}^{\left\lfloor \sigma_n^{2/3}(\widehat{N})\right\rfloor}\frac{q_n^k(kt)^2}{\sigma_n^2(\widehat{N})}- 
	\sum\limits_{k \in K_n^{+}}\frac{q_n^k(kt)^2}{\sigma_n^2(\widehat{N})}- \sum\limits_{k \in K_n^{-}}\frac{q_n^k(kt)^2}{\sigma_n^2(\widehat{N})}.
	\end{align*}
	Now for either of the two sets, $K_n^{\pm}$, notice that
$$
- \sum_{k\in K_n^{\pm}}k^2 q_n^{k} \ge -  \sum_{k\in K_n^{\pm}} \frac{k^2 q_n^{k}}{(1-q_n^{k})^2} = o(n^\frac{3}{2})
$$
where the last equality follows by definition. Therefore, we can see the missing contributions do not matter in the proof of (\ref{indfun2}) and rest of the proof follows as in Lemma \ref{locallimitthm} without any alteration to procedure. 

To prove (\ref{littleosigma2}), 
we would need to alter the count of the number of positive terms in the cosine.  However, the number of terms is not impacted enough to change its asymptotic order of magnitude.  We still restrict our sum to the interval
$$
\left\lceil\frac{1}{2}\sigma_n^{\frac{2}{3}}(\widehat{N})\right\rceil\le k \le \left\lfloor\sigma_n^{\frac{2}{3}}(\widehat{N})\right\rfloor.
$$
Because of the bounds on $k,$ it is clear that $1>q_n^k>c$ for the same $c$ defined in the proof of Lemma~\ref{locallimitthm}.  The number of missing terms
is then bounded above by   
$$
 \frac{1}{c}\sum_{k\in K_n^{\pm} } q_n^{k} \le  \frac{1}{c}\sum_{k\in K_n^{\pm} } \frac{k^2q_n^{k}}{(1-q_n^k)^2} = o(n^{\frac{3}{2}}). 
$$
Since the number of terms on this interval (with the excluded terms included) is $O(n^{\frac{3}{2}})$  the missing terms cannot have any meaningful contribution in the argument and therefore,  (\ref{littleosigma2}) follows as in Lemma \ref{locallimitthm}.

\end{proof}

Finally, we have the following theorem which allows us to consider the probability distribution $Q_{q_n}$ instead of $\mathbb{P}_n$.
\begin{theorem}
	\label{Frestthm}
	For all Borel sets $B \subseteq \mathbb{R}^{d_n^{+}+d_n^{-}}$ and $W_n$ as defined by~\eqref{defWn},
	$$\label{prohorovthm}
	\lim_{n \to \infty} \sup_{B} |\mathbb{P}_n(W_n^{-1}(B)) - Q_{q_n}(W_n^{-1}(B))| = 0\,.
	$$
\end{theorem}
\begin{proof}
	By Lemma~\ref{UBLem}, it suffices to prove Equation~\eqref{GTONE}.  Lemma~\ref{locallimitthm}, Lemma~\ref{locallimitthm2} and Equation~\eqref{algadj} prove that the the numerator and denominator of the left--hand--side of Equation~\eqref{GTONE} are asymptotically equivalent, so \eqref{GTONE} holds.
\end{proof} 

As per Fristedt's paper \cite{MR1094553}, we will now explicitly define
$$
K^{\pm}_n = \left\{k : k \geq \dfrac{\sqrt{3n}}{\pi}\left( \ln \dfrac{\sqrt{3n}}{\pi} - \ln(t_n) \right) \right\}\,,
$$
where $t_n$ any sequence growing to infinity that is $o(n^{1/4})$. Notice that~\eqref{smalloverK} holds for these $K^{\pm}_n$.

To conclude this section, note that a recent work of \cite{grabner2010general,ngointeger} provided straight forward analytic conditions for Fredist's conditioning device to hold.  It would be intriguing to see how these analytic conditions fit into this more general framework.

%----------------------------------------------------------------------------------
\section{Distributions of the perimeter, tilt and length}
\label{PTLsec}
%----------------------------------------------------------------------------------
In this section we compute the distributions of the perimeter, tilt and the length of a concave composition $(\lambda^-,c,\lambda^+)$, where $c=0$.  We begin with the perimeter, which is in correspondence with the length of the partition, since by Euler, the largest part of a partition is in bijection with the length of that partition (see  \cite{MR1634067}).
%We begin with the perimeter, which is in correspondence with the pair $(\lambda^-, \lambda^+)$, as given by a classic bijection due to Euler \cite{MR1634067}.  \mandy{I would say something more like this: We begin with the perimeter, which is in correspondence with the length of the partition, since by Euler, the largest part of a partition is in bijection with the length of that partition (see  \cite{MR1634067}).}

In light of Theorem \ref{Frestthm}, we need only consider the distribution of the perimeter over $Q_{q_n}$.

\begin{theorem}
	\label{perimeter_theorem}
	For all $n \in \mathbb{N}$, let 
	$$
	f_n(x) = \dfrac{\sqrt{3n}}{\pi}x + \dfrac{\sqrt{3n}}{\pi} \ln \dfrac{\sqrt{3n}}{\pi}\,.
	$$
	For fixed $x,y \in \mathbb{R}$,
	$$
	\lim_{n \to \infty} \hspace{0.02in} \mathbb{P}_n \left( \ell(\lambda^-) \leq f_n(x), \ \ell(\lambda^+) \leq f_n(y) \right) = e^{-(e^{-x}+e^{-y})}.
	%(\ell(\lambda^-), \ell(\lambda^+))  = \dfrac{\pi}{\sqrt{3n}} (x,y) + \dfrac{\pi}{\sqrt{3n}} \ln \dfrac{\sqrt{3n}}{\pi} (1,1) \right) = e^{-(x+y)}e^{-(e^{-x}+e^{-y})}
	$$
\end{theorem}

\begin{proof}
	From \cite{MR1634067}, we have that $(q;q)_j^{-1}$ generates partitions of $n$ whose length is at most $j$.  Therefore, 
	$$
	Q_q(\ell(\lambda^-) \leq i, \ \ell(\lambda^+) \leq j) = (q^{i+1};q)_{\infty}(q^{j+1};q)_{\infty}\,.
	$$
	Let $i=f_n(x)$, $j = f_n(y)$ and $q = q_n$.  From the $q$-Binomial Theorem \cite{MR1634067},
	\begin{equation}
	\label{qBinomThm}
	(z;q)_\infty^{-1} = \sum_{n=0}^\infty \frac{z^n}{(q;q)_n}.
	\end{equation}
	Plugging in $z=\tau e^{-x}$ and $q=e^{-\tau}$ and noting that $\lim_{\tau\to 0}\tau^n/(e^{-\tau};e^{-\tau})_n=1/n!$, we observe
	$$
	\lim_{\tau \to 0} (\tau e^{-x}; e^{-\tau})_{\infty} = e^{-e^{-x}}\,.
	$$
	Setting $q = q_n$, which tends to one, in Equation \eqref{qBinomThm} and applying Theorem \ref{Frestthm} completes the proof.
\end{proof}

\begin{lemma}
	\label{PerimeterQside}
	%{\color{red}{replace $f_{\tau}$ with $g_{\tau}$}}
	For $x \in \mathbb{R}$ and $0 < \tau < 1$, let $g_{\tau}(x) = (x- \ln \tau)/\tau$, and $q = e^{-\tau}$. 
	\begin{enumerate}[label=\Roman*.]
		\item For all $y \in \mathbb{R}$, 
		$$
		(\tau e^{-y}q; q)_{\infty} \leq e^{-qe^{-y}}\,.
		$$
		
		\item For all $x,y \in \mathbb{R}$,
		$$
		Q_{q}\left((\ell(\lambda^+),\ell(\lambda^{-}))= \left( g_{\tau}(x), g_{\tau}(y) \right) \right) \leq e^{-(x+y)}e^{-(e^{-x}+e^{-y})}\,.
		$$
		\item If $\tau e^{-2y}, \tau e^{-2x} \to 0$ as $\tau \to 0$, 
		\begin{align*}
		&Q_{q}\left((\ell(\lambda^+),\ell(\lambda^{-}))= \left( g_{\tau}(x), g_{\tau}(y) \right) \right)\\
		&= \tau^2 e^{-(x+y)}e^{-(e^{-x}+e^{-y})}\left(1+O(\tau\left(1+e^{-2y}\right))\right)\left(1+O(\tau\left(1+e^{-2x}\right))\right) \,.
		\end{align*}
	\end{enumerate}
\end{lemma} 

\begin{proof}
	From \cite{MR1634067} we have that $q^j (q;q)_j^{-1}$ generates partitions of $n$ whose length is $j$.  Therefore, 
	$$
	Q_q((\ell(\lambda^+),\ell(\lambda^{-}))=(a,b)) = q^{a+b}(q^{a+1};q)_\infty (q^{b+1};q)_\infty.
	$$
	By letting $a = g_{\tau}(x)$, and $b = g_{\tau}(y)$, we obtain 
	$$
	Q_q((\ell(\lambda^+),\ell(\lambda^{-}))=(g_{\tau}(x),g_{\tau}(y)))=\tau^2 e^{-x-y} (\tau e^{-x} q;q)_\infty (\tau e^{-y} q;q)_\infty.
	$$
	Now consider the expression $(\tau e^{-y}q;q)_{\infty}$.  Namely, 
	\begin{align*}
	\ln (\tau e^{-y}q;q)_\infty &= -\sum_{l=1}^\infty \frac{\tau^l e^{-ly} q^{l}}{l(1-q^l)} \\
	&= - \tau e^{-y} \frac{q}{1-q} + R\,,
	\end{align*}
	where $R$ is the remainder.  Since $R$ is negative, and $-\tau/(1 - q) < -1$, then 
	$$
	(\tau e^{-y}q;q)_{\infty} \leq e^{-qe^{-y}}\,,
	$$
	which proves $I$.  Furthermore,
	$$
	\tau^2 e^{-x-y}(\tau e^{-x} q;q)_\infty (\tau e^{-y} q;q)_\infty \leq \tau^2 e^{-x-y}e^{-q(e^{-x} + e^{-y})}\,,
	$$
	which shows $II$.
	
	To see $III$, suppose $\tau e^{-2y}, \tau e^{-2x} \to 0$ as $\tau \to 0$.  If $\tau e^{-2y} = o(1)$, then $e^{-ly} = o(1/\tau^{l/2})$.  Thus,
	\begin{align*}
	\sum_{l = 2}^{\infty} \dfrac{\tau^l e^{-ly} q^{l}}{l(1-q^l)} &= e^{-2y}\tau^2 \sum_{l = 0}^{\infty} \dfrac{\tau^l e^{-ly} q^{l+2}}{(l+2)(1-q^{l+2})} \\
	&\ll e^{-2y}\tau^2 \sum_{l = 0}^{\infty} \dfrac{\tau^{l/2} q^{l+1}}{(l+1)(1-q^{l+1})} \\
	&\leq e^{-2y}\tau \sum_{l=1}^{\infty} \dfrac{\tau^{\l/2}}{(l+1)^2} \\
	&\leq \tau e^{-2y} \dfrac{\pi^2}{6}\,,
	\end{align*}
	provided $\tau < 1$, and, using the assumption $q = e^{-\tau}$ to get $q^{l + 1}/(1-q^{l+1}) \leq 1/(l+1)$.
	%using the fact that $e^x - 1 > x$ for $x > 0$\mandy{Pawel was confused about second-to-last inequality. Maybe explain a bit more}.  
	Observe that $\tau e^{-2y} \pi^2/6 \to 0$ as $\tau \to 0$ and we require the estimate
	\begin{equation}
	\frac{\tau q}{1-q}= \tau\left(\frac{1}{1-q}-1\right) = 1 + O(\tau)\,,
	\end{equation}
	which provides us with 
	\begin{align*}
	\ln (\tau e^{-y}q;q)_\infty &= -\sum_{l=1}^\infty \frac{\tau^l e^{-ly} q^{l}}{l(1-q^l)} \\
	&= - \tau e^{-y} \frac{q}{1-q} + R \\
	&= - e^{-y} +O(\tau\left(1+e^{-2y}\right))\,.
	\end{align*}
	Now we use the bound $e^{s+t}= e^s\left(1+ O(t)\right)$ as $t\to 0$ to complete the proof.  Setting $s=-e^{-y}$ and $t= O(\tau\left(1+e^{-2y}\right))$ gives us
	$$
	(\tau e^{-y}q;q)_\infty = e^{-e^{-y}}\left(1+  O(\tau\left(1+e^{-2y}\right))\right)\,.
	$$
\end{proof}

We now move to the tilt, and the length of the concave composition $(\lambda^-,c,\lambda^+)$.  Given that the perimeter, which is the same as the length, is distributed as a pair of independent identically distributed extreme value distributions, it is not surprising that the length is the convolution of two Gumbel distributions while the tilt is logistically distributed.

\begin{theorem}\label{AviTheorem}
	For fixed $x \in \mathbb{R}$,   
	\begin{equation}
	\label{lengtheq}
	\small{\lim_{n \to \infty} \mathbb{P}_n \left( \ell(\lambda^-,c, \lambda^+) \leq \dfrac{\sqrt{3n}}{\pi}x + \dfrac{\sqrt{3n}}{\pi} \ln \left( \dfrac{\sqrt{3n}}{\pi} \right) \right) = e^{\frac{-x}{2}} \int_{-\infty}^{\infty} e^{-t} e^{ e^{\frac{-x}{2}} 2 \cosh\left( t \right)} dt}\,,
	\end{equation}
	%where $K_0(z)$ is the modified Bessel function.  
	%\[Q_q(l(\lambda^-)+l(\lambda^+) = \frac{x-2\ln \tau}{\tau}) \thicksim  2\tau e^{-x} K_0(2e^{x/2}) \] where $K_0(z)$ is the $K$ modified Bessel function.  
	and 
	\begin{equation}
	\lim_{n \to \infty} \mathbb{P}_n \left(  \ell(\lambda^-) - \ell(\lambda^+) \leq \dfrac{\sqrt{3n}}{\pi}x  \right) = \dfrac{1}{1 + e^{-x}}\,.
	\end{equation}
	%\[Q_q(l(\lambda^-)-l(\lambda^+) = \frac{x}{\tau}) \thicksim  \tau \frac{e^{-x}}{(1-e^{-x})^2}.  \] 
\end{theorem}
\begin{proof}  The proofs of the two equations are similar, and here we prove Equation \eqref{lengtheq}.  For any $z \in \mathbb{N}$, \cite{MR1634067} gives
	\begin{align*} 
	Q_q(\ell(\lambda^-) + \ell(\lambda^+) \leq z)  &= \sum_{j=0}^z Q_q(\ell(\lambda^{-})\leq z-j) \hspace{0.01in}Q_q(\ell(\lambda^{+})=j )\\
	&= \sum_{j=0}^z q^j(q^{z-j+1};q)_\infty (q^{j+1};q)_\infty \\
	&= q^{z/2} \sum_{j=-z/2}^{z/2} q^j(q^{z/2-j+1};q)_\infty (q^{z/2 + j+1};q)_\infty \,.
	\end{align*}
	Letting $z = {\sqrt{3n}x}/{\pi} + {\sqrt{3n}}/{\pi} \ln \left( {\sqrt{3n}}/{\pi} \right)$ and $q = q_n$ in the right-hand-side, we get
	\begin{align*}
	Q_q(\ell(\lambda^-) + \ell(\lambda^+) \leq z) =& \\
	&\hspace{-3cm} e^{\frac{-x}{2}}\frac{\pi}{\sqrt{3n}} \sum_{j=-z/2}^{z/2} e^{\frac{-\pi j}{\sqrt{3n}}} \left( \frac{e^{\frac{-x}{2}}\pi}{\sqrt{3n}} e^{\frac{-\pi j}{\sqrt{3n}}}q_n;q_n \right)_{\infty} \left( \dfrac{e^{\frac{-x}{2}}\pi}{\sqrt{3n}} e^{\frac{\pi j}{\sqrt{3n}}}q_n;q_n \right)_{\infty}\,.
	\end{align*}
	For any $M$ such that $M/\sqrt{n}$ diverges, we break up the sum on the right as
	$
	Q_q(\ell(\lambda^-) + \ell(\lambda^+) \leq z) =  \Sigma_1 + \Sigma_2  + \Sigma_3\,,
	$
	where
	$$
	\Sigma_1 = e^{\frac{-x}{2}}\frac{\pi}{\sqrt{3n}} \sum_{j=-z/2}^{-M-1} e^{\frac{-\pi j}{\sqrt{3n}}} \left( \frac{e^{\frac{-x}{2}}\pi}{\sqrt{3n}} e^{\frac{-\pi j}{\sqrt{3n}}}q_n;q_n \right)_{\infty} \left( \dfrac{e^{\frac{-x}{2}}\pi}{\sqrt{3n}} e^{\frac{\pi j}{\sqrt{3n}}}q_n;q_n \right)_{\infty}\,,
	$$
	$$
	\Sigma_2 = e^{\frac{-x}{2}}\frac{\pi}{\sqrt{3n}} \sum_{j=-M}^{M} e^{\frac{-\pi j}{\sqrt{3n}}} \left( \frac{e^{\frac{-x}{2}}\pi}{\sqrt{3n}} e^{\frac{-\pi j}{\sqrt{3n}}}q_n;q_n \right)_{\infty} \left( \dfrac{e^{\frac{-x}{2}}\pi}{\sqrt{3n}} e^{\frac{\pi j}{\sqrt{3n}}}q_n;q_n \right)_{\infty}\,,
	$$
	$$ 
	\Sigma_3 =  e^{\frac{-x}{2}}\frac{\pi}{\sqrt{3n}} \sum_{j=M+1}^{z/2} e^{\frac{-\pi j}{\sqrt{3n}}} \left( \frac{e^{\frac{-x}{2}}\pi}{\sqrt{3n}} e^{\frac{-\pi j}{\sqrt{3n}}}q_n;q_n \right)_{\infty} \left( \dfrac{e^{\frac{-x}{2}}\pi}{\sqrt{3n}} e^{\frac{\pi j}{\sqrt{3n}}}q_n;q_n \right)_{\infty}\,.
	$$
	Applying Lemma \ref{PerimeterQside} to $\Sigma_3$ gives us
	\begin{align*}
	\Sigma_3 &\leq e^{\frac{-x}{2}}\frac{\pi}{\sqrt{3n}} \sum_{j=M+1}^{z/2} e^{\frac{-\pi j}{\sqrt{3n}}} e^{-q_ne^{\frac{-\pi j}{\sqrt{3n}} - \frac{x}{2} }}  e^{-q_ne^{\frac{\pi j}{\sqrt{3n}} - \frac{x}{2} }} \\
	&= e^{\frac{-x}{2}}\frac{\pi}{\sqrt{3n}}  \sum_{j=M+1}^{z/2} e^{\frac{-\pi j}{\sqrt{3n}}} e^{-q_n e^{\frac{-x}{2}} 2 \cosh\left(\frac{\pi j}{ \sqrt{3n}}\right)} \\
	&\thicksim e^{\frac{-x}{2}} \int_{\frac{\pi(M+1)}{\sqrt{3n}}}^{\infty} e^{-t} e^{-2q_n e^{\frac{-x}{2} \cosh(t)}} dt \\
	& < e^{\frac{-x}{2}}  \int_{\frac{\pi (M+1)}{\sqrt{3n}}}^{\infty} e^{-t}e^{-2 q_n e^{\frac{-x}{2}}e^{-t}} dt = 
	\frac{-1}{2q_n}\left(e^{-2q_ne^{\frac{-x}{2}}e^{\frac{-\pi(M+1)}{\sqrt{3n}}}} - 1 \right)\,,
	\end{align*}
	which converges to zero as $n \rightarrow \infty$.  A similar argument shows that $\Sigma_1$ also converges to zero as $n \rightarrow \infty$.
	
	For $\Sigma_2$, the Euler-Maclaurin Formula says
	$$
	\Sigma_2 = \int_{-M-1}^{M} f(j) dj + \frac{1}{2}(f(M) - f(-M-1)) + R\,,
	\vspace{-0.1in}
	$$
	where
	$$
	|R| \leq \frac{1}{12} \int_{-M-1}^{M} |f^{(2)}(j)| dj \hspace{0.2in} \text{ and } \hspace{0.2in}  f(j) = e^{\frac{-x}{2}}\frac{\pi}{\sqrt{3n}} e^{\frac{-\pi j}{\sqrt{3n}}} e^{-q_n e^{\frac{-x}{2}} 2 \cosh\left(\frac{\pi j}{ \sqrt{3n}}\right)}\,.
	$$
	It is not difficult to show that $f(M) - f(-M-1)$ converges to zero as $n \rightarrow \infty$.  Splitting
	$$
	\int_{-M-1}^{M} |f^{(2)}(j)| dj = \int_{-M-1}^{0} |f^{(2)}(j)| dj + \int_{0}^{M} |f^{(2)}(j)| dj\,,
	$$
	it is not difficult to show that each of the two integrals on the right-hand-side here converges to zero as $n \rightarrow \infty$.  Next, the integral
	$$
	\int_{-M-1}^{M} f(j) dj = \int_{-\infty}^{\infty} f(j) dj\,,
	$$
	since each integral $\int_{-\infty}^{-M-1} f(j) dj$ and $\int_{M}^{\infty} f(j)dj$ converges to zero as $n \rightarrow \infty$.  Hence,
	$$
	\lim_{n \to \infty} \Sigma_2 = e^{\frac{-x}{2}} \int_{-\infty}^{\infty} e^{-t} e^{ e^{\frac{-x}{2}} 2 \cosh\left( t \right)} dt\,,
	\vspace{-0.08in}
	$$
	and using Theorem \ref{Frestthm} completes the proof.
\end{proof}

%------------------------------------------------------------------------------------------------
\section{The limiting graphical representation of a concave composition}
\label{graphsec}
%------------------------------------------------------------------------------------------------
%The last section of this paper, although still uses the Boltzmann measure on partitions, the techniques are a bit more implicit in nature.  

%We already have observe the distinction between $\mathbb{P}_n$ and $\mathbb{P}_n|c=0$ is purely formal as $n\to \infty$ as $c=0$ almost surely.  Hence for the rest of this section we consider only this case.  
In this section, we consider the graphical representation of a concave composition by applying the techniques on graphing partitions from \cite{MR2915644}.  Our departure begins by decomposing the Boltzmann measure further by looking at a weighted sum across the uniform measure on ordinary partitions of $k$, which we denote $\mu_k.$ So that the notation does not become cumbersome we assume $n$ is an even natural number going forward. The case of $n$ odd follows by the same methods.

The following lemma is not too surprising and decomposes the measure into its main components $v=(\lambda^-, c, \lambda^+)$ although we know $c=0$ a.s. as $n\to \infty$ by \eqref{ciszero}.

\begin{lemma} Let $n\in \mathbb{N},$ $S$ a set of concave compositions of $n,$ $\delta>0$ sufficiently small, and $\mu_k$ be the uniform measure on ordinary integer partitions of $k.$ There exists $w_k>0$, dependent on $n$, with $\sum_{k=0}^n w_k=1$ so that as $n\to \infty$, \\
	%\begin{align}\label{discretemodelequation}
	%	& \text{if } c = 0, \nonumber\\
	%	&\mathbb{P}_{n}(S)= \displaystyle \sum_{k=- n^{\frac{3}{4}+\delta}}^{n^{{\frac{3}{4}+\delta}}} \mu_{\frac{n}{2}+k}(\{\lambda^+:v\in S \})\cdot\mu_{\frac{n}{2}-k}(\{\lambda^-:v\in S \})\cdot w_{\frac{n}{2}+k} + O\left(\frac{1}{\sqrt{n}}\right)\nonumber\\
	%	&\text{and if } c > 0,~\mathbb{P}_{n}(S)= O\left(\frac{1}{\sqrt{n}}\right).
	%	\end{align}

	\begin{equation} \label{discretemodelequation}\mathbb{P}_{n}(S)= \begin{cases}\displaystyle \sum_{k=- n^{\frac{3}{4}+\delta}}^{n^{{\frac{3}{4}+\delta}}} \mu_{\frac{n}{2}+k}(\{\lambda^+:v\in S \})\cdot\mu_{\frac{n}{2}-k}(\{\lambda^-:v\in S \})\cdot w_{\frac{n}{2}+k} \\
	\hspace{6cm}+ O\left(\frac{1}{\sqrt{n}}\right)\text{ if } c = 0, & \\ 
	O\left(\frac{1}{\sqrt{n}}\right) \hspace{6.4cm} \text{if } c>0 .&  \\
	\end{cases} \end{equation}
\end{lemma}
\begin{proof} Let $S$ be any set of concave compositions with norm equal to $n$ and denote each $v\in S$ as $v=(\lambda^{-},c,\lambda^{+}).$  We first condition on $c=0$ and apply \eqref{ciszero}. Then we have
	\begin{align*}
	\mathbb{P}_{n}(S)&= \mathbb{P}_{n}(S|c=0) \mathbb{P}_{n}(c=0) 
	+ 
	\mathbb{P}_{n}(S|c>0) \mathbb{P}_{n}(c>0) \\
	&= \mathbb{P}_{n}(S|c=0) + O\left(\frac{1}{\sqrt{n}}\right).
	\end{align*}
	
	Conditioning on $c=0,$ any concave composition can be represented as a pair of partitions $v= (\lambda^+,\lambda^-)$ where $|\lambda^+| + |\lambda^-| = n.$  We can condition on the size of $|\lambda^+|$ and write 
	\[ \mathbb{P}_{n}\{S\}= \sum_{k=1}^{n-1} \mu_k(\{\lambda^+: v\in S\})\mu_{n-k}(\{\lambda^-: v\in S\}) w_k+O\left(\frac{1}{\sqrt{n}}\right)
	\] 
	where $w_k = p(k)p(n-k)/p_2(n)$ and $\mu_k(\lambda)=\delta_k(|\lambda|)/p(k).$ The number $p(k)$ counts the total number of partitions of $k$ and $\delta_k(n)$ is the Kronecker delta function.  All that is left to do is to  bound the ``tails" of the distribution.
	
	Recall the classic asymptotic \cite[Theorem 6.2]{MR1634067}
	\begin{equation}\label{hardyram}
	p(n)= \frac{1}{4n\sqrt{3}}\exp\left(\pi \sqrt{\frac{2n}{3}}\right)\left(1+O\left(\frac{1}{\sqrt{n}}\right)\right).
	\end{equation}
	%This estimate along with (\ref{p2approx}) proves for $\left| \begin{align}w_k  
	%&=\frac{\sqrt[4]{12} n^\frac{5}{4} }{4\sqrt{6} k(n-k) }\exp\left(\pi \sqrt{\frac{2}{3}} \left(\sqrt{k} + \sqrt{n-k} -\sqrt{2n}  \right)\right)\left(1 + O\left(\frac{1}{\sqrt{n}} \right) \right).\label{wkexpansion}
	%\end{align}
	%Expanding the exponent near $k=n/2$  we observe, 
	%\begin{equation}\label{expansion}
	%\pi \sqrt{\frac{2}{3}} \left(\sqrt{k} + \sqrt{n-k} -\sqrt{2n}  \right) = -  \frac{\pi }{\sqrt{3} n^\frac{3}{2} }\left(k- \frac{n}{2}\right)^2 + O\left(\frac{\left(k-\frac{n}{2} \right)^3}{ n^\frac{5}{2} } \right).
	%\end{equation}  
	%Substituting  \eqref{expansion} into \eqref{wkexpansion}
	%\begin{align} w_{\frac{n}{2}+z}  &\thicksim \frac{1}{\sqrt{2\pi } \sigma_n(\widehat N)}\exp\left(-\frac{1}{2}\left(\frac{z}{\sigma_n(\widehat N)}\right)^2\right), \label{concentrationinequality} \end{align}
	% where $\sigma_n(\widehat N) = \sqrt[4]{\frac{3}{4 \pi^2}}n^\frac{3}{4}.$ Note that $z\in \mathbb{Z}$ for $n$ even and $z\in \mathbb{Z}+\frac{1}{2}$ if $z$ is odd.
	
	Using \eqref{hardyram},
	$$
	p\left(\frac{n}{2}+z\right)p\left(\frac{n}{2}-z\right)\ll \exp\left( \frac{\pi}{\sqrt{3}} \left(\sqrt{n+2z} + \sqrt{n-2z}  \right)\right).
	$$  Consider the function on the interval $z\in [0,\frac{n}{2}]$
	$$
	f(z) = \sqrt{n+2z} + \sqrt{n-2z}  .
	$$
	This function has derivative
	$$
	f'(z)= \frac{1}{\sqrt{n+2z}}-\frac{1}{\sqrt{n-2z}} \le 0.
	$$
	and is zero only when $z=0.$ For $z>n^{\frac{3}{4}+\delta}$ we have the uniform bound
	\begin{align}
	p\left(\frac{n}{2}-z\right)p\left(\frac{n}{2}+z\right) &\le \exp\left( \frac{\pi}{\sqrt{3}} \left(\sqrt{n+2n^{\frac{3}{4}-\delta}} + \sqrt{n-2n^{\frac{3}{4}-\delta}}  \right)\right)  \\
	&= \exp\left(\frac{2\pi}{\sqrt{3}}\sqrt{n} -\frac{\pi}{\sqrt{3}}n^{2\delta}\right).
	\end{align} Hence, for $\left|z\right|\geq n^{\frac{3}{4}+\delta}$ we can bound $w_k$ by a second application of \eqref{hardyram}$$
	w_{\frac{n}{2}+z} \ll n e^{-n^{2\delta}}.
	$$
	Since $\mu_k$ is a probability measure $\mu_k(S)\le 1$ for every set $S$ of partitions and we have for any set $S$
	$$
	\left(\sum_{k=n^{\frac{3}{4}-\delta}}^n + \sum_{k=-n^{\frac{3}{4}-\delta}}^{-n}\right) \mu_{\frac{n}{2}+k}(S)\mu_{\frac{n}{2}-k}(S) w_{\frac{n}{2}+k}\ll n^2 e^{-n^{2\delta}} \ll\frac{1}{n^2}.
	$$ 
\end{proof}

A celebrated work of H. N. V. Temperley is on the Ferrers graph, or Young diagram, of a large positive integer $n$.  In particular, \cite{Temp52} shows that these graphs have a somewhat uniform shape to them, which we call the limit shape, defined by the curve 
\begin{equation}
\label{partlimshape}
%\exp \left(-\frac{\pi }{\sqrt{6}}x\right) + \exp\left(-\frac{\pi }{\sqrt{6}}y\right)=1.
e^{-\frac{\pi }{\sqrt{6}}x} + e^{-\frac{\pi }{\sqrt{6}}y}=1.
\end{equation}  
Figure \ref{partitionLimitCurve} shows this limit curve along with a normalized random partition of a positive integer.  

\begin{figure}[h]
	\includegraphics[scale=.8]{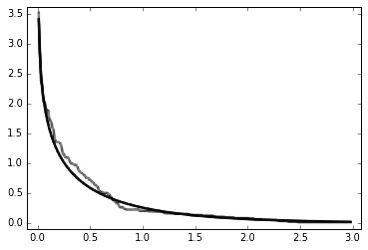}
	\caption{The graph of a normalized random partition of a positive integer in the order of $10^{5}$ along with the limit shape given by Equation \eqref{partlimshape}.}
	\label{partitionLimitCurve}
\end{figure}
There have been many proofs which show that the curve given by Equation \eqref{partlimshape} is a limit shape curve for the uniform statistic on the Young diagrams.  The work of Kerov and Vershik \cite{VK85} says that a proof was obtained by Vershik using the results of Szalai and Tur{\'a}n from \cite{ST77}.
%One work of Kerov and Vershik, \cite{VK85}, gives a proof of that fact using the results of Szalai and Tur{\'a}n from \cite{ST77}.  
A later, independent work of Vershik \cite{Ver95,Ver96} gives a proof from the point of view of $Q_q$.  The work of Petrov \cite{Pet09} shows an elementary proof from the point of view of $\mathbb{P}_n$.

To define a Young diagram, it is more convenient to define $x$ as a function of $y$ through the function $X_m(\lambda)$ which counts the number of $m$'s in the partition $\lambda.$  Under this convention, the Young diagram of $\lambda$ is defined as the graph of the function 
\[
x_{\lambda}(y) =\sum_{m>y}X_{m}(\lambda).
\]
Precisely, by \cite[Theorem 8, Combining Equation (41) and Equation (40)]{MR2915644} we know that there exist $\gamma, \delta>0,$ so that for all $y\in \mathbb{R}$ there exists $C>0$ so that for $\epsilon>0$ small and $k$ large,
%\begin{align}
%\label{limitshape}
%\mu_k\left\{\left|\frac{1}{\sqrt{k}} x_\lambda (\sqrt{k}y)-\frac{\sqrt{6}}{\pi }\ln \left(1 -e^{-\frac{\pi }{\sqrt{6}}y}\right) \right|>\epsilon \right\}< &\\ 
%&\hspace{-0.5in} k^\gamma \exp \left(-C(y)\epsilon^2\sqrt{k^{1-\delta}}\right) \nonumber \,, 
%\end{align} 
\begin{equation}
\label{limitshape}
\mu_k\left\{\left|\frac{1}{\sqrt{k}} x_\lambda (\sqrt{k}y)-\frac{\sqrt{6}}{\pi }\ln \left(1 -e^{-\frac{\pi }{\sqrt{6}}y}\right) \right|>\epsilon \right\} < k^\gamma \exp \left(-C\sqrt{k^{1-\delta}}\right).\, 
\end{equation}

Concave compositions have a similar property with an important caveat.  There is no unique limit shape for a concave composition but rather a family of curves which will asymptotically fit the graphical representations. For a concave composition $v=(\lambda^{-1},c,\lambda^{+})$ with $|v|=n$  we assign a curve 
\[
v\to \frac{\pi C^{v}(x)}{\sqrt{3 n}}e^{\frac{\pi}{\sqrt{3}}|x|} +e^{-\frac{\pi}{\sqrt{3}}y} = 1\,, 
\]
for some $C^{v}(x) = (C_+^{v}) {\bf{1}}_{x> 0}(x) + (C_{-}^{v}) {\bf{1}}_{x< 0}(x)$ that is a stepwise function given by the indicator functions.  For most concave compositions, these constants are small in size and behave like a pair of i.i.d. log Gumbel distributions $C^{\nu}_{\pm}.$    
 
We now construct the graphical representation of a concave composition by adapting the setup of \cite{MR2915644}.  Our first pass at the construction will be heuristic but we will then follow it up with a formal proof. The graph of $(\lambda^-,c,\lambda^+)$ is constructed by first drawing the central part $c$ as a step function that is centered at the origin.  Next, we draw simple functions which represent the graphs of $\lambda^-$ and $\lambda^+$ to the left and right of $c$, respectively.  The resulting picture should always look like a stepwise approximation to a convex function.  See Figure \ref{ConcaveCompositionEx} for an example.

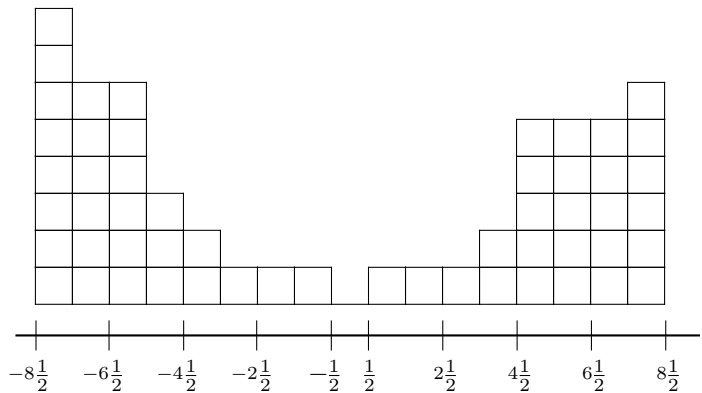
\begin{figure}[h]
	$$
	%\hspace{-0.01in}
	\Large\tableau[scY]{ \cr \cr & & & \bl & \bl & \bl & \bl & \bl & \bl & \bl & \bl & \bl & \bl & \bl & \bl & \bl & \cr & & & \bl & \bl & \bl & \bl & \bl & \bl & \bl & \bl & \bl & \bl & & & & \cr & & & \bl & \bl & \bl & \bl & \bl & \bl & \bl & \bl & \bl & \bl & & & & \cr & & & & \bl & \bl & \bl & \bl & \bl & \bl & \bl & \bl & \bl & & & & \cr & & & & & \bl & \bl & \bl & \bl & \bl & \bl & \bl & & & & & \cr & & & & & & & & \fr[b] & & & & & & & &}
	$$
	\begin{center}
	\begin{tikzpicture}[xscale=1.6]
	
	\draw [thick]  (1,0) -- (7.95,0);
	\draw (1.23,-0.2) -- (1.23,0.2);
	\draw (1.97,-0.2) -- (1.97,0.2);
	\draw (2.72,-0.2) -- (2.72,0.2);
	\draw (3.48,-0.2) -- (3.48,0.2);
	\draw (4.23,-0.2) -- (4.23,0.2);
	\draw (4.61,-0.2) -- (4.61,0.2);
	\draw (5.35,-0.2) -- (5.35,0.2);
	\draw (6.11,-0.2) -- (6.11,0.2);
	\draw (6.86,-0.2) -- (6.86,0.2);
	\draw (7.61,-0.2) -- (7.61,0.2);
	
	\node[align=left, below] at (1.275,-0.2) {$\frac{1}{2}$};
	\node[align=left, below] at (1.06,-0.3) {$\scriptstyle{-8}$};
	\node[align=left, below] at (2.035,-0.2) {$\frac{1}{2}$};
	\node[align=left, below] at (1.82,-0.3) {$\scriptstyle{-6}$};
	\node[align=left, below] at (2.795,-0.2) {$\frac{1}{2}$};
	\node[align=left, below] at (2.58,-0.3) {$\scriptstyle{-4}$};
	\node[align=left, below] at (3.555,-0.2) {$\frac{1}{2}$};
	\node[align=left, below] at (3.34,-0.3) {$\scriptstyle{-2}$};
	\node[align=left, below] at (4.26,-0.2) {$\frac{1}{2}$};
	\node[align=left, below] at (4.1,-0.275) {$-$};
	\node[align=left, below] at (4.62,-0.2) {$\frac{1}{2}$};
	\node[align=left, below] at (5.45,-0.2) {$\frac{1}{2}$};
	\node[align=left, below] at (5.33,-0.3) {$\scriptstyle{2}$};
	\node[align=left, below] at (6.21,-0.2) {$\frac{1}{2}$};
	\node[align=left, below] at (6.09,-0.3) {$\scriptstyle{4}$};
	\node[align=left, below] at (6.97,-0.2) {$\frac{1}{2}$};
	\node[align=left, below] at (6.85,-0.3) {$\scriptstyle{6}$};
	\node[align=left, below] at (7.73,-0.2) {$\frac{1}{2}$};
	\node[align=left, below] at (7.61,-0.3) {$\scriptstyle{8}$};
	\end{tikzpicture}
		\end{center}
	\caption{Graphical representation of the concave composition $(8,6,6,3,2,1,1,1,0,1,1,1,2,5,5,5,6).$ }
	\label{ConcaveCompositionEx}
\end{figure}

It is useful to define the ``tick marks" at which the concave composition increases in $y$ value.  Classically this would just be $x_{\lambda}(y)$, but our right partition, $\lambda^+$, must be flipped.  Furthermore, both the partitions $\lambda^-$ and $\lambda^+$ are shifted by half a unit.  The resulting ``tick marks" are
\begin{equation}\label{gdef}
g_{\lambda^\pm}(y) =\pm\left(\ell(\lambda^\pm)-x_{\lambda^\pm}(y) +\frac{1}{2}\right).
\end{equation}

For each $\lambda= \lambda^+, \lambda^-$, we can define the function
\[
G_{\lambda}(x) = \sum_{i=0}^\infty i \cdot {\bf{1}}_{(g_{\lambda}(i-1),g_{\lambda}(i)]}(x).
\]
Setting $v = (\lambda^-,c, \lambda^+)$, the graphical representation is the sum of the three functions
\[
G_v(x) =G_{\lambda^+}(x)+ G_{\lambda^-}(x) + c\cdot{\bf{1}}_{(-1/2,1/2]}(x).
\]
Observe that since the sum of all parts is $n$, then $\int_{\mathbb{R}}G_v(x)dx=n.$  Thus, we normalize the graph by dividing by $n.$ Concurrently, we shrink the graph by a $\sqrt{n}$ factor in the $x$ and $y$ directions. The result is  \begin{align*} \widetilde{G}_v(x) &= \frac{\sqrt{n}}{n } G_v(\sqrt{n}x).\end{align*} By letting $y_i=i/\sqrt{n}$ we can observe $\sqrt{n}x\in (g_{\lambda}(i-1),g_{\lambda}(i)]$ if and only if  $x\in \left(\frac{g_{\lambda}(\sqrt{n}y_{i}-1)}{\sqrt{n}},\frac{g_{\lambda}(\sqrt{n}y_{i})}{\sqrt{n}}\right]$ and $\widetilde{G}_v(x) =y_i$ on this interval.  For $x>0,$ we have reflected the Young diagram around the line $x=\ell(\lambda^+)/\sqrt{n}$.  Thus for $x>0$, we expect
$$
e^{-\frac{\pi}{\sqrt{3}}\left(\frac{\ell(\lambda^{+})}{\sqrt{n}}-x\right)}+ e^{-\frac{\pi}{\sqrt{3}}y}=1\,,
$$
and for $x<0$ we reflect the formula for $\lambda^{-}$ to get
$$
e^{-\frac{\pi}{\sqrt{3}}\left(\frac{\ell(\lambda^{-})}{\sqrt{n}}+x\right)}+ e^{-\frac{\pi}{\sqrt{3}}y}=1.
$$
Theorem \ref{perimeter_theorem} provides us guidance as to how to think of $\ell(\lambda^{\pm})$ 
$$
-\frac{\pi}{\sqrt{3n}}\ell(\lambda^{\pm}) =\ln\left(\frac{\pi}{\sqrt{3n}}\right) +A_{\pm}(v)\,,
$$
where
$$
\lim_{n \to \infty} \mathbb{P}_n(A_{+} < u,\ A_{-}< v) = e^{-e^{-u}-e^{-v}.}
$$ 
This motivates our function
$$
C^v(x) = e^{-A_{+}} {\bf{1}}_{x> 0}(x) +e^{-A_{-}}{\bf{1}}_{x< 0}(x) \,,
$$
so that
$$
\frac{\pi C^v(x)}{\sqrt{n}}e^{\frac{\pi}{\sqrt{3n}}|x|} + e^{-\frac{\pi}{\sqrt{3n}}y}=1.
$$
Our narrative so far has been somewhat heuristic, so the following is a more formal approach. Recall that 
$
\widetilde{G}_v (x) =y 
$
if and only if \\ \begin{equation}\label{Xvdefinition}\sqrt{n} x\in \left(g_{\lambda^{+}}\left(\lfloor y\sqrt{n}\rfloor-1\right),g_{\lambda^{+}}\left(\lfloor y\sqrt{n}\rfloor\right)\right]\end{equation} for $x>0$. For simplicity we fix $y>0$ and then consider the random variable $X_v>0$ as the any such $X_v$ be a function of the concave composition's graph such that equation \eqref{Xvdefinition} is satisfied.  This function $X_v$ defines a sequence of random variables on the sequence of measures $\mathbb{P}_n$ as $n\to \infty.$ Of course by construction, there is a second case where $X_v<0$ but this is similar by symmetry. In the case of the limit shape this $X_v$ converges to a constant in probability.  We will find that $X_v$ converges in distribution, a weaker form of convergence, to a random variable.
\begin{theorem}  
	\label{DanTheorem}
	Fix $y>0,$ let $v=(\lambda^{-},c,\lambda^{+})$ be a concave composition, and denote $X_v$ satisfying \eqref{Xvdefinition}.  Then the following limits hold with respect to $\mathbb{P}_n$ as $n \to \infty$
	\begin{enumerate}[label=\Roman*.]
		\item $\frac{\pi g_{\lambda^+}(y \sqrt{n})}{\sqrt{3n}}-\frac{\pi}{\sqrt{3}}X_v$ vanishes in probability.
		\item $\frac{\pi g_{\lambda^+}(y \sqrt{n})}{\sqrt{3n}}- \ln \left(\frac{\sqrt{3n}}{\pi}\left(1-e^{-\frac{\pi y }{\sqrt{3}}}\right)\right)$ converges in distribution.
		\item $\frac{\pi X_v}{\sqrt{3}}- \ln \left(\frac{\sqrt{3n}}{\pi}\left(1-e^{-\frac{\pi y }{\sqrt{3}}}\right)\right)$ converges in distribution.
	\end{enumerate}
\end{theorem}
\begin{proof} To prove Part I, we observe that $
	y\geq\widetilde{G}_v (X_v) 
	$ implies $X_v \le   \frac{\pi}{\sqrt{3}n}\left(g_{\lambda^{+}}\left(\lfloor y\sqrt{n}\rfloor\right)+1\right)$ and

	$$
		\frac{\pi}{\sqrt{3}n}g_{\lambda^{+}}\left(\lfloor y\sqrt{n}\rfloor-1\right)-\frac{\pi}{\sqrt{3}n}g_{\lambda^{+}}\left(\lfloor y\sqrt{n}\rfloor\right) < X_v- \frac{\pi}{\sqrt{3}n}g_{\lambda^{+}}\left(\lfloor y\sqrt{n}\rfloor\right).
	$$
	Recall  Equation \eqref{gdef}
	$$	
	g_{\lambda^{+}}(y) = \ell(\lambda^+)-x_{\lambda^{+}}(y)+ \frac{1}{2}.
	$$
	Using Equation \eqref{gdef},
	\begin{align}
	\frac{\pi}{\sqrt{3}n}g_{\lambda^{+}}\left(\lfloor y\sqrt{n}\rfloor-1\right)-&\frac{\pi}{\sqrt{3}n}g_{\lambda^{+}}\left(\lfloor y\sqrt{n}\rfloor\right)\label{diffx}\\
	&=\frac{\pi}{\sqrt{3}n}x_{\lambda^{+}}\left(\lfloor y\sqrt{n}\rfloor\right)-\frac{\pi}{\sqrt{3}n}x_{\lambda^{+}}\left(\lfloor y\sqrt{n}\rfloor-1\right)\nonumber.
	\end{align}
	So if we are able to demonstrate convergence in probability of the right hand side of \eqref{diffx} to zero then we would be done.  In particular it suffices to show that for all $\epsilon>0,$ if we let $E$ denote the event
	$$
	E:=\left\{(\lambda^{-},c,\lambda^{+}): \left|\frac{1}{\sqrt{n}} x_{\lambda^{+}} (\sqrt{n}y)- \frac{\sqrt{3}}{\pi }\ln \left(1-e^{-\frac{\pi }{\sqrt{3}}y}\right) \right|>4\epsilon \right\}
	$$
	and if we can show that
	\begin{equation}
	\label{anotherDanEq}
	\lim_{n\to \infty} \mathbb{P}_n\left(E\right)=0\,
	\end{equation}
	then \eqref{diffx} vanishes in probability. Here we are invoking the squeeze theorem for convergence in probability.
	
	If the norm of $\lambda^{+}$ were a priori forced to diverge as $n$ grows large, then showing Equation \eqref{anotherDanEq} is trivial in light of Equation \eqref{discretemodelequation}.  However, this requires some work to justify by using Equation \eqref{discretemodelequation}.  That is, we show Equation \eqref{limitshape} holds with respect to $\mu_k$ when we replace $k$ with $n/2$ uniformly for some $\delta>0$ and all $k\in\left[n/2 - n^{\frac{3}{4}+\delta},n/2 + n^{\frac{3}{4}+\delta}\right]$.  Let 
	$$
	B_k = \left\{\lambda \vdash k : \left|\frac{1}{\sqrt{k}} x_{\lambda} (\sqrt{k}y)-\frac{\sqrt{6}}{\pi }\ln \left(1 -e^{-\frac{\pi }{\sqrt{6}}y}\right) \right|\le\epsilon \right\}
	$$
and we define $B = \bigcup\limits  _{l = -n^{\frac{3}{4}+\delta}}^{n^{\frac{3}{4}+\delta}} B_{n/2 + l}.$
	We need to apply Equation (\ref{limitshape}), and notice that for $\eta>0,$ $n$ sufficiently large and $k>n/3$
	\begin{equation}
	\mu_k B_k^{C}< \left(\frac{n}{3}\right)^\gamma \exp \left(-C(y)4\epsilon^2\sqrt{\left(\frac{n}{3}\right)^{1-\delta}}\right)<\eta
	\end{equation}
	and by how $\mu_k$ is defined $\mu_k B_j^C = 0$ unless $j=k.$  Therefore $\mu_k B^C <\eta$ uniformly.
	
	%\mandy{Question for Dan: clarify on k in the previous sentence.}  %This allows our limit to hold in probability as $n \to \infty.$   
	Now we apply the mean value theorem to give an estimate
	$$
	\frac{1}{\sqrt{k}} = \frac{1}{\sqrt{\frac{n}{2}}} - \frac{1}{2\xi(k)^\frac{3}{2}}\left(k-\frac{n}{2}\right)
	$$
	for some $\xi(k) \in \left[n/2 - n^{\frac{3}{4}+\delta},n/2 + n^{\frac{3}{4}+\delta}\right].$  For $n$ sufficiently large we then can state that $\xi(k)\geq \frac{n}{3}.$ For every $\lambda \in B$ and $y \in \mathbb{R}$, 
	\begin{align*}
	\frac{x_{\lambda}(\sqrt{\frac{n}{2}}y)}{\sqrt{\frac{n}{2}}} 
	&= \frac{x_{\lambda}(\sqrt{\frac{n}{2}}y)}{\sqrt{k}} + O\left(\frac{x_{\lambda}(\sqrt{\frac{n}{2}}y)}{n^{\frac{3}{4}-\delta}}\right). 
	\end{align*}
	Invoking $\lambda^+ \in B$ we have
	$$
	\left|\frac{1}{\sqrt{k}} x_{\lambda^+} \left(\sqrt{\frac{n}{2}}y\right)-\frac{\sqrt{6}}{\pi }\ln \left(1 -e^{-\frac{\pi }{\sqrt{6}}\sqrt{\frac{2k}{n}}y}\right) \right|\le \epsilon.
	$$
	Using the triangle inequality, we then can state  $x_{\lambda^+} (\sqrt{\frac{n}{2}}y)\ll \sqrt{n}$ as $k\le n.$ Therefore for $n$ sufficiently large, 

	\begin{align} \label{eq1}
	\left|\frac{x_{\lambda^+}(\sqrt{\frac{n}{2}}y)}{\sqrt{\frac{n}{2}}} - \frac{x_{\lambda^+}(\sqrt{\frac{n}{2}}y)}{\sqrt{k}}\right|
	&\ll \frac{1}{n^{\frac{1}{4}-\delta}}< \epsilon.
	\end{align}

	Also given $\lambda^+ \in B$, we let $u=\frac{\sqrt{n}y}{\sqrt{2k}}$ and we note that $u\to y$ uniformly on compact subsets of $y$ as $n\to \infty$

	\begin{align}\label{eq2}
	\left|\frac{x_{\lambda^+}(\sqrt{\frac{n}{2}}y)}{\sqrt{k}} - \frac{x_{\lambda^+}(\sqrt{k}y)}{\sqrt{k}}\right|&= \left|\frac{x_{\lambda^+}(\sqrt{\frac{n}{2}}u)-x_{\lambda^+}(\sqrt{k}y)}{\sqrt{k}} \right|\nonumber\\
	&\le\left|\frac{\sqrt{6}}{\pi }\ln\left(\frac{1-e^{-\frac{ \pi}{\sqrt{6}} y}}{1-e^{-\frac{ \pi}{\sqrt{6}} u}}\right)\right|+ 2\epsilon \\
	&\le 3 \epsilon. \nonumber
	\end{align}

	Combining \eqref{eq1} and \eqref{eq2} and using the triangle inequality we have now demonstrated that if $\lambda^+ \in B$ then any concave composition $(\lambda^-, 0,\lambda^{+})\notin E$. This translates to our measure on partitions as $\mu_k(\{\lambda^{+}: v \in E^C\}|B)=1.$ By conditioning on $B$ we get
	\begin{align*}
	\mu_k(\{\lambda^{+}: v \in E^C\}) &= \mu_k(\{\lambda^{+}: v \in E^C\}|B) \mu_k(B) + \mu_k(\{\lambda^{+}: v \in E^C\}|B^C)\mu_k(B^C)\\
&\geq \mu_k(B) \geq 1- \eta.  
	\end{align*}
	Taking the complementation demonstrates $\mu_k(\{\lambda^{+}: v \in E\}) \le \eta.$ Recall that $\eta>0$ does not depend on $k$ so plugging in this bound into  \eqref{discretemodelequation} and taking the limit proves
	\begin{equation} \lim_{n\to \infty}\mathbb{P}_n\left(E\right) \le\eta. \label{mukbound} \end{equation} 
	Since this upper bound holds for every $\eta>0,$ the axiom of completeness allows us to conclude that it indeed vanishes proving the first result.

	For Part II, we must recall Slutsky's  theorem.  This theorem states that if $X_n\to X$ in distribution and $Y_n\to Y$ in probability then $X_n+Y_n\to X+Y$ in distribution. Theorem \ref{perimeter_theorem} has already shown 
	\[
	\lim_{n\to \infty}\mathbb{P}_n \left\{\frac{\pi\ell(\lambda^+)}{\sqrt{3n}} - \ln \left(\frac{\sqrt{3n}}{\pi}\right)<x \right\}=  e^{-e^{-x}},
	\]
	which is convergence in distribution.  Using Equation \eqref{gdef} completes the proof. Part III follows immediately from the first and second parts.
	%We must therefore prove that $x_{\lambda^{\pm}}(y),$ when normalized, converges in probability to a constant (dependent on $y$ of course).  
	
\end{proof}

\begin{figure}[h]
	\includegraphics[scale=.8]{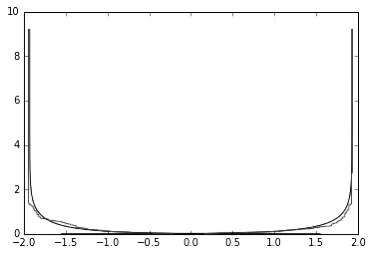}
	\caption{The graph of a normalized random concave composition, with central part $c=0$, of a large positive integer in the order of $10^{4}$ along with the proposed limit shape given by Theorem \ref{DanTheorem}. The sampling was done with respect to $Q_{q_n}$.  }
\end{figure}
%------------------------------------------------------------------------------------------------
\section{Future work}
\label{futurework}
%------------------------------------------------------------------------------------------------
We conclude with some open questions and possible threads of future work. First of all, we have assumed $c=0$ throughout this paper but it would be interesting to derive the distribution of $c$ or allow $c=\alpha\sqrt{n}$ for some small $\alpha\in(0,1)$, close to $0$. In addition, the questions we have answered about concave compositions can also be asked of other compositions, such as strongly concave compositions (which was mentioned in \cite{MR3152010}) or convex compositions. Another interesting direction would be to consider different measures on concave compositions, such as the Plancherel measure or the Ewens measure which have both been defined on partitions.  Specifically, see if the asymptotic bounds for the perimeter, tilt and length are tighter under these and other measures.

\begin{acknowledgements}
The authors would like to thank George Andrews, Pawe{\l} Hitczenko and Anatoly Vershik for their wonderful insights and helpful comments. \end{acknowledgements}

% Authors must disclose all relationships or interests that 
% could have direct or potential influence or impart bias on 
% the work: 
%
% \section*{Conflict of interest}
%
% The authors declare that they have no conflict of interest.

% BibTeX users please use one of
%\bibliographystyle{spbasic}      % basic style, author-year citations
\bibliographystyle{spmpsci}      % mathematics and physical sciences
\bibliography{Reflist}   % name your BibTeX data base

\def\cprime{$'$}
\begin{thebibliography}{10}
\providecommand{\url}[1]{{#1}}
\providecommand{\urlprefix}{URL }
\expandafter\ifx\csname urlstyle\endcsname\relax
  \providecommand{\doi}[1]{DOI~\discretionary{}{}{}#1}\else
  \providecommand{\doi}{DOI~\discretionary{}{}{}\begingroup
  \urlstyle{rm}\Url}\fi

\bibitem{MR1634067}
Andrews, G.E.: The Theory of Partitions.
\newblock Cambridge Mathematical Library. Cambridge University Press, Cambridge
  (1998).
\newblock Reprint of the 1976 original

\bibitem{And11}
Andrews, G.E.: Concave compositions.
\newblock Electron. J. Combin. \textbf{18}(2), Paper 6, 13 (2011)

\bibitem{And13}
Andrews, G.E.: Concave and convex compositions.
\newblock Ramanujan J. \textbf{31}(1-2), 67--82 (2013).
\newblock \doi{10.1007/s11139-012-9394-6}.
\newblock
  \urlprefix\url{http://dx.doi.org.ezproxy.lib.uwf.edu/10.1007/s11139-012-9394-6}

\bibitem{MR3152010}
Andrews, G.E., Rhoades, R.C., Zwegers, S.P.: Modularity of the concave
  composition generating function.
\newblock Algebra Number Theory \textbf{7}(9), 2103--2139 (2013).
\newblock \doi{10.2140/ant.2013.7.2103}.
\newblock \urlprefix\url{http://dx.doi.org/10.2140/ant.2013.7.2103}

\bibitem{MR2180794}
Bender, E.A., Canfield, E.R.: Locally restricted compositions. {I}.
  {R}estricted adjacent differences.
\newblock Electron. J. Combin. \textbf{12}, Research Paper 57, 27 pp.
  (electronic) (2005).
\newblock
  \urlprefix\url{http://www.combinatorics.org/Volume_12/Abstracts/v12i1r57.html}

\bibitem{MR2994899}
Bryson, J., Ono, K., Pitman, S., Rhoades, R.C.: Unimodal sequences and quantum
  and mock modular forms.
\newblock Proc. Natl. Acad. Sci. USA \textbf{109}(40), 16063--16067 (2012).
\newblock \doi{10.1073/pnas.1211964109}.
\newblock \urlprefix\url{http://dx.doi.org/10.1073/pnas.1211964109}

\bibitem{MR1235434}
Chaganty, N.R., Sethuraman, J.: Strong large deviation and local limit
  theorems.
\newblock Ann. Probab. \textbf{21}(3), 1671--1690 (1993).
\newblock
  \urlprefix\url{http://links.jstor.org/sici?sici=0091-1798(199307)21:3<1671:SLDALL>2.0.CO;2-2&origin=MSN}

\bibitem{MR1667320}
Corteel, S., Pittel, B., Savage, C.D., Wilf, H.S.: On the multiplicity of parts
  in a random partition.
\newblock Random Structures Algorithms \textbf{14}(2), 185--197 (1999).
\newblock \doi{10.1002/(SICI)1098-2418(199903)14:2<185::AID-RSA4>3.3.CO;2-6}.
\newblock
  \urlprefix\url{http://dx.doi.org/10.1002/(SICI)1098-2418(199903)14:2<185::AID-RSA4>3.3.CO;2-6}

\bibitem{MR1094553}
Fristedt, B.: The structure of random partitions of large integers.
\newblock Trans. Amer. Math. Soc. \textbf{337}(2), 703--735 (1993).
\newblock \doi{10.2307/2154239}.
\newblock \urlprefix\url{http://dx.doi.org/10.2307/2154239}

\bibitem{MR1924786}
Goh, W.M.Y., Hitczenko, P.: Average number of distinct part sizes in a random
  {C}arlitz composition.
\newblock European J. Combin. \textbf{23}(6), 647--657 (2002).
\newblock \doi{10.1006/eujc.2002.0435}.
\newblock \urlprefix\url{http://dx.doi.org/10.1006/eujc.2002.0435}

\bibitem{MR2422389}
Goh, W.M.Y., Hitczenko, P.: Random partitions with restricted part sizes.
\newblock Random Structures Algorithms \textbf{32}(4), 440--462 (2008).
\newblock \doi{10.1002/rsa.20191}.
\newblock \urlprefix\url{http://dx.doi.org/10.1002/rsa.20191}

\bibitem{grabner2010general}
Grabner, P., Knopfmacher, A., Wagner, S.: A general asymptotic scheme for
  moments of partition statistics.
\newblock preprint  (2010)

\bibitem{MR2531482}
Heubach, S., Mansour, T.: Combinatorics of Compositions and Words.
\newblock Discrete Mathematics and its Applications (Boca Raton). CRC Press,
  Boca Raton, FL (2010)

\bibitem{MR0203748}
Lo{\`e}ve, M.: Probability Theory.
\newblock Third edition. D. Van Nostrand Co., Inc., Princeton, N.J.-Toronto,
  Ont.-London (1963)

\bibitem{MR2417935}
MacMahon, P.A.: Combinatory Analysis. {V}ol. {I}, {II} (bound in one volume).
\newblock Dover Phoenix Editions. Dover Publications, Inc., Mineola, NY (2004).
\newblock Reprint of {{\i}t An introduction to combinatory analysis} (1920) and
  {{\i}t Combinatory analysis. Vol. I, II} (1915, 1916)

\bibitem{ngointeger}
Ngo, T.H., Rhoades, R.C.: Integer partitions, probabilities and quantum modular
  forms.
\newblock preprint  (2014)

\bibitem{Pet09}
Petrov, F.: Two elementary approaches to the limit shapes of {Y}oung diagrams.
\newblock Zap. Nauchn. Sem. S.-Peterburg. Otdel. Mat. Inst. Steklov. (POMI)
  \textbf{370}(Kraevye Zadachi Matematicheskoi Fiziki i Smezhnye Voprosy Teorii
  Funktsii. 40), 111--131, 221 (2009).
\newblock \doi{10.1007/s10958-010-9845-9}.
\newblock
  \urlprefix\url{http://dx.doi.org.ezproxy.lib.uwf.edu/10.1007/s10958-010-9845-9}

\bibitem{ST77}
Szalay, M., Tur{\'a}n, P.: On some problems of the statistical theory of
  partitions with application to characters of the symmetric group. {I}.
\newblock Acta Math. Acad. Sci. Hungar. \textbf{29}(3-4), 361--379 (1977)

\bibitem{Temp52}
Temperley, H.N.V.: Statistical mechanics and the partition of numbers. {II}.
  {T}he form of crystal surfaces.
\newblock Proc. Cambridge Philos. Soc. \textbf{48}, 683--697 (1952)

\bibitem{Ver95}
Vershik, A.M.: Asymptotic combinatorics and algebraic analysis.
\newblock In: Proceedings of the {I}nternational {C}ongress of
  {M}athematicians, {V}ol.\ 1, 2 ({Z}\"urich, 1994), pp. 1384--1394.
  Birkh\"auser, Basel (1995)

\bibitem{Ver96}
Vershik, A.M.: Statistical mechanics of combinatorial partitions, and their
  limit configurations.
\newblock Funktsional. Anal. i Prilozhen. \textbf{30}(2), 19--39, 96 (1996).
\newblock \doi{10.1007/BF02509449}.
\newblock
  \urlprefix\url{http://dx.doi.org.ezproxy.lib.uwf.edu/10.1007/BF02509449}

\bibitem{VK85}
Vershik, A.M., Kerov, S.V.: Asymptotic of the largest and the typical
  dimensions of irreducible representations of a symmetric group.
\newblock Funktsional. Anal. i Prilozhen. \textbf{19}(1), 25--36, 96 (1985)

\bibitem{MR0229604}
Wright, E.M.: Stacks.
\newblock Quart. J. Math. Oxford Ser. (2) \textbf{19}, 313--320 (1968)

\bibitem{MR0282940}
Wright, E.M.: Stacks. {II}.
\newblock Quart. J. Math. Oxford Ser. (2) \textbf{22}, 107--116 (1971)

\bibitem{MR0299575}
Wright, E.M.: Stacks. {III}.
\newblock Quart. J. Math. Oxford Ser. (2) \textbf{23}, 153--158 (1972)

\bibitem{MR2915644}
Yakubovich, Y.: Ergodicity of multiplicative statistics.
\newblock J. Combin. Theory Ser. A \textbf{119}(6), 1250--1279 (2012).
\newblock \doi{10.1016/j.jcta.2012.03.002}.
\newblock \urlprefix\url{http://dx.doi.org/10.1016/j.jcta.2012.03.002}

\end{thebibliography}

% Non-BibTeX users please use
%\begin{thebibliography}{}
%
% and use \bibitem to create references. Consult the Instructions
% for authors for reference list style.
%
%\bibitem{Reflist}
% Format for Journal Reference
%Author, Article title, Journal, Volume, page numbers (year)
% Format for books
%\bibitem{RefB}
%Author, Book title, page numbers. Publisher, place (year)
% etc
%\end{thebibliography}

\end{document}